\documentclass[11pt,reqno]{amsart}
\usepackage{amssymb}
\usepackage{mathrsfs}
\PassOptionsToPackage{dvipsnames}{xcolor}
\usepackage{pgfplots}
\pgfplotsset{compat=1.11}
\usepackage{amsbsy}
\usepackage{amsfonts}
\usepackage{amsmath}
\usepackage{amsthm}
\usepackage{amssymb}
\usepackage{ytableau}
\usepackage[all,arc]{xy}
\usepackage{tikz-cd}
\usepackage{enumerate}

\newcommand{\KKK}{K}
\newcommand{\RRR}{R}
\newcommand{\SSS}{S}
\newcommand{\II}{\mathscr{I}}
\newcommand{\NN}{\mathbb{N}}
\newcommand{\ordr}{\operatorname{ord}}

\newcommand{\PC}{\text{C}}

\newcommand{\ddR}{K}
\newcommand{\ddS}{U}
\newcommand{\dda}{d}
\newcommand{\ddnu}{\alpha}
\newcommand{\ddB}{\mathfrak{b}}

\newcommand{\DR}{R} 
\newcommand{\DD}{D}
\newcommand{\DDS}{{\tilde{R}}} 

\newcommand{\QQ}{\mathbb{Q}}

\newcommand{\ZZ}{\mathbb{Z}}

\newcommand{\ivr}{\alpha}

\newcommand{\aaa}{\mathfrak{a}} 

\newcommand{\bb}{\mathfrak{b}}
\newcommand{\cc}{\mathfrak{c}}
\newcommand{\dd}{\mathfrak{d}}
\newcommand{\ff}{\mathfrak{c}}
\newcommand{\mm}{\mathfrak{m}}
\newcommand{\pp}{\mathfrak{p}} 
\newcommand{\qq}{\mathfrak{q}}
\newcommand{\tp}{{t}}

\newcommand{\sA}{\mathcal{A}}

\newcommand{\sD}{\mathcal{D}}

\newcommand{\sT}{\mathcal{T}}

\newcommand{\ct}{{\bf ct}} 
\newcommand{\Cl}{{\rm Cl}}

\newcommand{\T}{\mathcal{T}} 

\newcommand{\PP}{\mathbb{P}}
\newcommand{\order}{\operatorname{\it v}} 
\newcommand{\iv}{\it v}
\newcommand{\dnu}{\operatorname{\it v}_t}
\newcommand{\nuB}{\operatorname{\it v}_\mathfrak{b}}

\newcommand{\Spec}{\operatorname{\mathcal{P}}}

\newcommand{\Prim}{{\rm Prim}}

\newcommand{\sI}{\mathscr{I}}
\theoremstyle{plain}
\newtheorem{theorem}{Theorem}[section]
\newtheorem{prop}[theorem]{Proposition}
\newtheorem{corollary}[theorem]{Corollary}
\newtheorem{lemma}[theorem]{Lemma}

\newtheorem{thm}[theorem]{Theorem} 
\newtheorem{lem}[theorem]{Lemma}

\theoremstyle{definition}
\newtheorem{definition}[theorem]{Definition}

\newtheorem{example}[theorem]{Example}

\newtheorem{defn}[theorem]{Definition} 

\theoremstyle{remark}
\newtheorem{remark}[theorem]{Remark}
\newtheorem{rem}[theorem]{Remark}

\numberwithin{equation}{section}

\definecolor {UMblue}  {RGB}{0, 39, 76}
\definecolor {UMmaize} {RGB}{255, 203, 5}

\definecolor {color_b}{RGB}{255,0,0}
\definecolor {color_c}{RGB}{20, 200, 30}
\definecolor {color_a}{RGB}{0,0,255}

\definecolor{lgreen} {RGB}{180,210,100}
\definecolor{dblue}  {RGB}{20,66,129}
\definecolor{ddblue} {RGB}{11,36,69}
\definecolor{lred}   {RGB}{220,0,0}
\definecolor{nred}   {RGB}{224,0,0}
\definecolor{norange}{RGB}{230,120,20}
\definecolor{nyellow}{RGB}{255,221,0}
\definecolor{ngreen} {RGB}{98,158,31}
\definecolor{dgreen} {RGB}{78,138,21}
\definecolor{nblue}  {RGB}{28,130,185}
\definecolor{jblue}  {RGB}{20,50,100}

\title[$B$-orderings and generalized factorials]{$B$-orderings for all ideals $B$ of Dedekind domains and generalized factorials} 
\author{Jeffrey C. Lagarias}
\address{Department of Mathematics, University of Michigan, Ann Arbor, MI 48109--1043, USA.}
\email{lagarias@umich.edu} 
\author{Wijit Yangjit} 
\email{yangjit@umich.edu}
\date{February 26, 2025}

\begin{document}

\begin{abstract}
This paper extends Bhargava's theory of $\mathfrak{p}$-orderings of subsets $S$ of a Dedekind ring $R$ valid for prime ideals $\mathfrak{p}$ in $R$. Bhargava's theory defines for integers $k\ge1$ invariants of $S$, the generalized factorials $[k]!_S$, which are ideals of $R$. This paper defines $\mathfrak{b}$-orderings of subsets $S$ of a Dedekind domain $D$ for all nontrivial proper ideals $\mathfrak{b}$ of $D$. It defines generalized integers $[k]_{S,\mathcal{T}}$, as ideals of $D$, which depend on $S$ and on a subset $\mathcal{T}$ of the proper ideals $\mathscr{I}_D$ of $D$. It defines generalized factorials $[k]!_{S,\mathcal{T}}$ and generalized binomial coefficients, as ideals of $D$. The extension to all ideals applies to Bhargava's enhanced notions of $r$-removed $\mathfrak{p}$-orderings, and $\mathfrak{p}$-orderings of order $h$.
\end{abstract}


%Polynomial rings, integer-valued polynomials 

\maketitle
%\tableofcontents
\setlength{\baselineskip}{1.0\baselineskip}

%%%%%%%%%%%%%%%%%%%%%%%%%%%%%%%%%%%%%%%%%%%%%%%%%%%%%%%%%
%\Notation: $\binom{k +\ell}{\ell}_{S, \sT, h}^{ \{ r\}}$  versus $\binom{k +\ell}{\ell}_{S, \sT}^{h,  \{ r\} }$ 
%%for ``order-h" and ``r-removed" and ``ring subset S"  and ``ideal-subset $\sT$", and, for generalized factorials
%$n !_{S, \sT, h}^{ \{ r\}  }$ versus $n!_{S, \sT}^{h, \{ r \} }$ . Both of these quantities are integral ideals in
%a Dedekind ring $R$.}
%%%%%%%%%%%%%%%%%%%%%%%%%%%%%%%%%%%%%%%%%%%%%%%%%%%%%%%%%
%%%%%%%%%%
%
% Section 1
%
%%%%%%%%%%
%%%%%%%%%%

\section{Introduction}\label{sec:1}
This paper extends Bhargava's theory of $\pp$-orderings and $\pp$-sequences, indexed by
prime ideals $\pp$, for arbitrary subsets $\SSS$ of a Dedekind domain $\DD$, to $\bb$-orderings and $\bb$-sequences,
 for all ideals $\bb$, for arbitrary subsets
$\SSS$ of a Dedekind domain $\DD$.

In 1996 Bhargava \cite{Bhar:96} developed  a theory of 
$\pp$-orderings and $\pp$-sequences, indexed by prime ideals $\pp$,  for arbitrary subsets $\SSS$ 
of  a class of commutative rings $\DR$ 
 called   Dedekind-type rings. 
 A  {\em Dedekind-type ring } is a Noetherian, locally principal ring in which all nonzero primes are maximal; it 
 may have zero divisors. Any Dedekind-type  ring is either a Dedekind domain or is a quotient of a Dedekind domain by a nonzero ideal. 
 Bhargava's  $\pp$-sequences are 
invariants of the set $\SSS$, given as an infinite sequence  of ideals $\nu_k(\SSS, \pp) = \pp^{\alpha_k(\SSS, \pp)} $ of $\DR$.
With these invariants he defined generalized factorials $k!_{\SSS}$ and generalized binomial coefficients
 ${ {k+ \ell}\choose{k}}_{\SSS}$ associated to the set $\SSS$ as ideals of $\DR$. 
 For Dedekind domains $\DD$ these   invariants 
have applications in specifying rings of polynomials $f(x) \in K[x]$, where $K$ is the quotient field of $\DD$,
that are integer-valued on the domain $\SSS$.
 
 In 2009 Bhargava   \cite{Bhar:09} introduced
  refined  invariants for arbitrary subsets $\SSS$ of a fixed Dedekind domain $\DD$
   associated 
 to various generalized types of  refined $\pp$-orderings and associated  $\pp$-sequences,
 defined for each nonzero prime ideal $\pp$ of $\ddR$. These  refined invariants depend on extra parameters
 $r \in \NN$ and $h \in \NN \cup \{ +\infty\}$.
The most general of these are {\em $r$-removed $\pp$-orderings of order $h$};
 the original $\pp$-orderings correspond to the choice $r=0$ and $h= +\infty$.
Bhargava used them to define corresponding generalized factorials $k!_{\SSS}^{h, \{r\}}$, 
and generalized binomial coefficients
${ {k+ \ell}\choose{k}}_{\SSS}^{h, \{r\}}$. 
%%%%%%%%%%%%%%%%%%%%%%%%%%%%%%%%%%%%%%%%%%%%%%
%Bhargava's original definition of $\pp$ orderings correspond
%to the parameter choice  $r=0$ and $h= +\infty$. 
%of these generalizations.
%corresponds to Bhargava's
%original definition of $\pp$-orderings. 
%%%%%%%%%%%%%%%%%%%%%%%%%%%%%%%%%%%%%%%%%%%%%%
He applied the generalized invariants to 
characterize rings of polynomials that are integer-valued 
under differencing operations on
a given subset of a Dedekind domain  $\DD$, and
to characterize bases of rings of smooth functions on compact subsets of local fields, giving orthonormal bases
of Banach spaces of locally analytic functions in terms of generalized binomial polynomials. 
 
 The extensions of this paper show  that Bhargava's  (additive) definitions of refined $\pp$-orderings
 used in \cite{Bhar:09},  give well-defined invariants 
 for  all ideals $\bb$ of a Dedekind domain $\DD$, rather than just prime ideals. 
 To formulate results, we use in place of  
 Bhargava's $\bb$-sequences, which are  ideals of $\DD$,
 the equivalent data of  {\em $\bb$-exponent sequences}, which are 
   the exponents of Bhargava's $\bb$-sequence ideals  as powers of $\bb$,
  and are nonnegative integers (or $+\infty$).

A first result is the  extension to 
all ideals of Dedekind domains of Bhargava's original $\pp$-orderings.

%%%%%%%%%%
%
% Main Theorem 1.1 (Special case)
%
%%%%%%%%%%
\begin{theorem}[Well-definedness of the $\bb$-exponent sequence of $S$ in Dedekind domains]\label{thm:well-definedness0}
Let $\DD$ be a Dedekind domain. 
Then for all 
ideals $\bb$ of $\DD$ and 
all nonempty subsets $\SSS$  of the ring $\DD$, if  $\mathbf{a}_1$ and $\mathbf{a}_2$ are 
%%%%%%%%%%%%%%%%
%``plain" $\bb-orderings
%%%%%%%%%%%%%%%%
 $\bb$-orderings of $S$,
 then 
 $\ivr_i\left(S,\bb,\mathbf{a}_1\right)=\ivr_i\left(S,\bb,\mathbf{a}_2\right)$ for all $i=0,1,2,\dots$.
\end{theorem}

\noindent The special case of $\DD= \ZZ$ of this theorem was proved in \cite{LY:24a}.

The main result of the paper is the well-definedness of $\bb$-exponent sequences for Dedekind domains,
for  all types of Bhargava's  generalized  orderings, which include two extra parameters $(h ,\{r\})$.
It is stated as   Theorem \ref{thm:NEW-MAIN} in Section \ref{sec:2}.
Using Theorem   \ref{thm:NEW-MAIN}  we define  generalized factorials and generalized binomial coefficients,  as fractional ideals of the ring $\DD$, 
and later show they are integral ideals. We deduce  divisibility properties for such generalized coefficients.

 The existence of 
  $\bb$-exponent sequence invariants for all ideals $\bb$
  inserts an extra degree of freedom $\sT$  in  Bhargava's theory. 
  Given  an arbitrary (nonempty) subset $\SSS$ of a Dedekind domain $\DD$,
and an arbitrary collection $\sT$ of  nonzero proper ideals   of $\DD$
 it permits defining  generalized factorials $(n)_{S, \sT}$, in a partially factorized form, with the definition,
 \begin{equation*}\label{def:factorials-S-T}
 [n]!_{S, \sT}^{h,r}= \prod_{\bb \in \sT} \bb^{  \alpha_n^{h, \{r\} }(S,  \bb)}.
 \end{equation*}
 (Only finitely many terms in the product
 are not the unit ideal, see Lemma \ref{lem:60}.)
We  define  generalized binomial
 coefficients  ${k+\ell \brack\ell}_{S,\sT}^{h, \{r\}} $, for $k , \ell \ge 0$. 
 The binomial coefficients are initially defined as fractional ideals in the ring $\DD$, and then shown  to be  integral ideals.
We establish  divisibility properties under  set inclusions in  the parameter $\sT$, and expect similar properties to hold under
set inclusion in the parameter $S$; see Section \ref{sec:8}. 
The  $\bb$-ordering  invariants are also defined for the  zero ideal $\bb=(0)$ and unit ideal $\bb=(1)$ cases, these cases always
contribute factors of $(1)$ or $(0)$; thus $\bb$-ordering invariants are  defined for all ideals in $\DD$.

 Bhargava's $\pp$-invariants  have important meanings as invariants of rings that describe
which functions on $\SSS$ extend to polynomial functions on $\DD$.  The general $\bb$-sequence invariants
shown to exist in Theorem \ref{thm:NEW-MAIN} do not, to our knowledge,  come with 
any parallel  combinatorial interpretation  related to polynomial functions.
It  remains a  topic for further investigation to find structural interpretations of these invariants. 

%%%%%%%%%%%%%%%%%%%%%%%%%%%%%%
%
% Subsection 1.1 Methods
%
%%%%%%%%%%%%%%%%%%%%%%%%%%%%%%%% 

\subsection{Methods} 

The proofs establish well-definedness of  the $\bb$-sequences  for $\SSS$ for Dedekind domains  $\DD$ by 
giving a mapping from the Dedekind domain $\DD$ to a local ring  $\KKK[[t]]$, with  $\KKK$ a field, 
where they can be identified with $\mm$-sequences of the image set $U$ for the maximal ideal $\mm$
of the local ring,  a case already covered by Bhargava's theory. Such maps are defined for a  {\em fixed} ideal $\bb$,
 but are required to   {\em simultaneously work for {all} subsets $\SSS$} of $\DD$.
The  mappings constructed do not respect either of the ring operations: ring  addition and ring multiplication. 
They  do preserve the $\bb$-ordering property,  mapping it to  an $\mm$-ordering   property and
with  identical $\bb$-exponent sequence invariants  and $\mm$-exponent sequence invariants. 
The mappings are nonconstructive in the general case and
depend on the Axiom of Choice (Note that Dedekind domains may have arbitrary cardinality.)

The proofs reduce the well-definedness of $\bb$-orderings for nonzero proper ideals $\bb$ 
to a special  case of   Bhargava's for refined  $\pp$-orderings, for
the maximal ideal $\mm= t \KKK[[t]]$ of a power series ring  over a field $\KKK[[t]]$. 
We state Bhargava's special case for  $\KKK[[t]]$ as Theorem \ref{thm:T-invariant}, and use it as a black box.
In \cite{LY:24a} we gave  an alternate proof for this special case for  ``plain $t$-orderings"($r=0$, $h=+\infty$)  based on a more 
analytical viewpoint (\cite[Theorem 4.5]{LY:24a}).
%%%%%%%%%%%%%%%%%%%%%%%%%%%%%%%%%%%%%%%%
% It is possible to extend this special case proof to ``r-removed $t$-orderings".
%%%%%%%%%%%%%%%%%%%%%%%%%%%%%%%%%%%%%%%%%

%%%%%%%%%%%%%%%%%%%%%%%%%%%%%%
%
% Subsection 1.2: Prior Work
%
%%%%%%%%%%%%%%%%%%%%%%%%%%%%%%% 

\subsection{Prior work}

The  original  Bhargava $\pp$-ordering construction, restricted to prime ideals, 
 has been extended in many directions, including multivariable polynomial generalizations, see 
 Evrard \cite{Evrard:12}, Rajkumar et al. \cite{RRP:18}.
 For applications of the refined $p$-orderings to rings of integer-valued
 polynomials, see Bhargava et al   \cite{BharCY:09}. For extensions of refined $p$-orderings
 in defining generalized regular sets in local fields, see Chabert et al \cite{CEF:13}.
 For use of refined $p$-orderings on $\ZZ$
 for the set $S=\PP$ of prime numbers, see Chabert \cite{Chab:15}.
 For methods for computing $r$-removed $p$-orderings and $p$-orderings of order $h$ on $R= \ZZ$,
 see Johnson \cite{John10}.
%However the extension to non-prime ideals has not been considered by other 

  Bhargava's original $\pp$-ordering definitions given in \cite{Bhar:96},  \cite{Bhar:00} (see Definition \ref{def:p-ordering} of this paper)  do not
  generalize  to composite ideals. 
  Counterexamples to the  $\bb$-sequenceability property
 for the  definition $\ivr_k^{\ast}(S, \bb, \mathbf{a})$ in \eqref{defn:123}
 for composite bases over $\ZZ$ are given in \cite[Example 3.4]{LY:24a}. 
   In this paper we show that Bhargava's refined $\pp$-ordering definitions in  \cite{Bhar:09} do extend to work for all ideals $\bb$.
  The paper  \cite{Bhar:09}  did  not raise  the possibility of such  extensions, and
  the proofs for well-definedness of refined $\pp$-ordering invariants  given there
  %in \cite{Bhar:09} 
   use the prime ideal property.

The  paper \cite{LY:24a} of the authors  constructed  
$(b)$-orderings   for arbitrary subsets $\SSS$ of the integers $\ZZ$, for all ideals $(b)$ 
of $\ZZ$. The case $R=\ZZ$ has two
simplifying features, compared to  the general Dedekind domain $\DD$.
First,  $\ZZ$ is a principal ideal domain, and the construction in Section \ref{subsec:proofs-general}
of this paper  is not needed.
Second, 
for $R=\ZZ$ the the multiplicative monoid  $\sI(\ZZ)$ of all ideals of $\ZZ$ under ideal multiplication
 is isomorphic to
the multiplicative monoid of nonnegative elements of $\ZZ$, and there are only two units $\pm 1$,
permitting generalized factorials and binomial coefficients 
to be directly defined as ring elements of $\ZZ$, rather  than as  integral ideals of $\ZZ$.  
These two features permit the generalized factorials and binomial
coefficients in the special case $S=\ZZ$  to be identified with classical factorials and binomial coefficients.
In addition  for   the $\ZZ$-case  the paper \cite{LY:24a} 
gave  for each fixed ideal $(b)$ 
 a direct construction of  the  associated congruence-preserving maps  
 in terms of polynomials encoding base $b$ digit expansions.
 The paper \cite{LY:24a} is a revised and augmented version
 of a chapter in  the second author's PhD thesis \cite{{Yangjit:22}}.
% avoiding the Axiom of Choice

%The method of proof is to reduce to 

%%%%%%%%%%%%%%%%%%%%%%%%%%%%%%
%
% Subsection1.2
%
%%%%%%%%%% %%%%%%%%%%%%%%%%%%%%
 \subsection{Contents }\label{subsec:12}
Section \ref{subsec:11}  reviews Bhargava's theory of $\pp$-orderings.
Section \ref{sec:2} presents the modified definition of $\bb$-orderings and $\bb$-sequences,
for subsets $\SSS$ of a Dedekind domain, which parallels Bhargava's definition in \cite{Bhar:09},  and states  the main results. 
 The principal result  is the well-definedness of the  $\bb$-sequence invariants
 for Dedekind domains,  Theorem \ref{thm:NEW-MAIN}.  
 We apply it  to define  generalized factorials, generalized binomial coefficients,  and generalized integers
 as fractional ideals of  $\DD$, and subsequently prove they are integral ideals,
 %in Section \ref{sec:6}. 
 The proofs follow in Sections \ref{sec:Bhargava-K[[t]]} to \ref{sec:6}.
% At the end of Section \ref{sec:2} we outline  the  contents of the remainder of the paper.
 %The proofs follow in subsequent sections.

%%%%%%%%%
% Subsec. 1.3
%%%%%%%%%%%%
\subsection{Notation}\label{subsec:13}

In this paper   $R$ denotes a general commutative ring with unit, and  $\sI(R)$ denotes the set of all
 % nonzero proper ideals 
 ideals of   $R$. 
 The set $\sI^{\ast}(R):= \sI(R) \smallsetminus \{(0)\}$ denotes 
 the set of all nonzero ideals of $R$.
 %$\sI^{\ast\ast}(R) := \sI(R)\smallsetminus \{ (0), R\}$ the set of
 %all nonzero proper ideals of $R$.
 %(A proper ideal $\bb$ is one with $\bb \ne \RRR$.)
 The notation    $\DD$ denotes a Dedekind domain, and $\ddR$ denotes a field, of  any characteristic.
%%%%%%%%%%
%
% Section 2
%
%%%%%%%%%%

\section{ Bhargava's $\pp$-orderings and generalized factorials}\label{subsec:11}
% formerly \subsec{11} 

We review Bhargava's theory of $\pp$-orderings, as given in \cite{Bhar:96}, \cite{Bhar:00}, \cite{Bhar:09}.  
We suppose $\DR$ is a Dedekind ring, i.e. either a Dedekind domain or a quotient of a Dedekind domain
by a  nonzero proper ideal. 
(A proper ideal $\aaa$ is one with $\aaa \subsetneq  \DD$.)

%%%%%%%%%%%%%%%%%%%%%%%%%%%%%%
%
% Subsection 2.1
%
%%%%%%%%%% 
 \subsection{Bhargava's generalized factorials}\label{subsec:21}
  Bhargava's generalized factorials are associated to nonempty sets $S$ of elements of $\DR$ and 
  (a subset $\sT$ of) 
    the set of all nonzero prime ideals of $\DR$, denoted
\begin{equation}
\PP\left(\DR\right)=\{\pp:\pp\text{ is a nonzero prime ideal in }\DR\}.
\end{equation}
Bhargava's generalized factorials $k!_S$ are ideals in $\DR$.

For each nonzero prime ideal $\pp$ in $\DR$ and a nonempty subset $S$ of $\DR$, he assigned an associated $\pp$-sequence 
$\left(\nu_k(S,\pp)\right)_{k=0}^\infty$ of $S$ in which $\nu_k(S,\pp)$ is a power of $\pp$. 
We can write
\begin{equation}
\nu_k(S,\pp)=\pp^{\ivr_k^{\ast}(S,\pp)},
\end{equation}
where $\ivr_k^{\ast}(S,\pp)\in\mathbb{N}\cup\{\infty\}$, with the conventions $\pp^0=\DR$ and $\pp^\infty=(0)$.
He constructed the associated $\pp$-sequence using $\pp$-orderings of $S$.

%%%%%%%%%
% Defn. 2.1
%%%%%%%%
%We describe it for the case $\DR=\mathbb{Z}$. 
%Let $p\in\PP$ be the prime number that generates $\pp=(p) \in\Spec(\mathbb{Z})$.
\begin{defn}\label{def:p-ordering}
 A \emph{$\pp$-ordering of $S$} is any sequence $\mathbf{a}=\left(a_i\right)_{i=0}^\infty$ of elements of $S$ in
 a Dedekind ring $\DR$ that can be formed recursively as follows:
\begin{enumerate}
\item[(i)] $a_0\in S$ is chosen arbitrarily;
\item[(ii)] Given $a_j\in S$, $j=0,\dots,i-1$, the next element $a_i\in S$ is chosen so that it minimizes the highest power of $\pp$ dividing the product $\prod_{j=0}^{i-1}\left(a_i-a_j\right)$.
\end{enumerate}

\end{defn} 

We note that:
\begin{enumerate}
\item[(1)] $\pp$-orderings of $S$ exist.
\item[(2)] This construction does not give a unique $\pp$-ordering of $S$ if $\vert S\vert>1$.
\item[(3)] A $\pp$-ordering of $S$ does not need to include all the elements of $S$.
\end{enumerate}

Bhargava defines the 
ideal  $\nu_i(S,\pp,\mathbf{a})$  (depending on the $\pp$-sequence $\mathbf{a}$) to be the highest power of $\pp$ dividing $\prod_{j=0}^{i-1}\left(a_i-a_j\right)$. 
Bhargava calls the sequence of ideals  $\left(\nu_i(S,\pp,\mathbf{a})\right)_{i=0}^\infty$ the \emph{$\pp$-sequence of $S$ associated to the $\pp$-ordering $\mathbf{a}$}. 
We write
\begin{equation}\label{eq:nu-to-alpha}
\nu_i(S,\pp,\mathbf{a})=\pp^{\ivr_i^{\ast}(S,\pp,\mathbf{a})},
\end{equation}
where
\begin{equation}\label{eq:Bharhava-defn}
\ivr_i^{\ast}(S,\pp,\mathbf{a}):=\order_\pp\left(\prod_{j=0}^{i-1}\left(a_i-a_j\right)\right)
\end{equation}
%%%%%%%
% \ivr = \alpha
%%%%%%
and $\order_\pp(\cdot)$ is
% the additive $\pp$-adic valuation 
given by
\begin{equation}
\order_\pp(a):=\sup\left\{\alpha\in\mathbb{N}:\pp^\alpha\text{ divides } a \DR \right\} \in \NN \cup \{ +\infty\}. 
\end{equation}
Here $\NN= \{ 0, 1, 2, ...\}$ denotes the set of nonnegative integers.
%When $R$ is  a Dedekind domain, then  $\order_\pp(\cdot)$ is an additive $\pp$-adic valuation. 

%Bhargava's key idea is the construction of $\pp$-orderings of $S$ for any fixed prime ideal $\pp$ of $\DR$. 
The minimization condition in Definition \ref{def:p-ordering}  can be written 
\begin{equation}\label{eqn:p-ordering0}
\order_{\pp}\left( \prod_{j=0}^{i-1} \left(a_i-a_j\right)\right) =\min_{s\in \SSS}\,\order_{\pp} \left( \prod_{j=0}^{i-1}\left(s-a_j\right) \right).
\end{equation}

Bhargava \cite[Theorem 1]{Bhar:97a}
 %\cite[Theorem~5]{Bhar:00} 
 proved the following main result:

%%%%%%%%%%
%
%  Theorem 2.2
%
%%%%%%%%%%

\begin{thm}[Well-definedness of the $\pp$-sequence of $S$]\label{thm:11}
%\\cite{Bhar:00}
The  $\pp$-sequence $\left(\nu_i(S,\pp,\mathbf{a})\right)_{i=0}^\infty$
of a subset $S$ of a Dedekind ring $\DR$ is independent of the choice of $\pp$-ordering $\mathbf{a}$.
\end{thm}

By Theorem \ref{thm:11},  one may write $\nu_i(S,\pp)=\nu_i(S,\pp,\mathbf{a})$ 
as an invariant under the choice of $p$-ordering $\mathbf{a}$ and call 
$\left(\nu_i(S,\pp)\right)_{i=0}^\infty$ the associated {\em $p$-sequence} of $S$.

The generalized factorials of $S$, denoted $k!_S$, are defined as in \cite[Definition~7]{Bhar:00} by
\begin{equation*}
k!_S:=\prod_{\pp\in\PP\left(\DR\right)}\nu_k(S,\pp),
\end{equation*}
where $\PP(\DR)$ denotes the set of all nonzero prime ideals of $\DR$.
All but finitely many terms in the product are the unit ideal $\DD$, so the right side is a well-defined integral ideal in $\DD$.
More generally, for subsets $\sT \subseteq \PP(\DR)$ of nonzero prime ideals one can define generalized factorials
\begin{equation}
k!_{S, \sT} :=\prod_{\pp\in\sT }\nu_k(S,\pp).
\end{equation}
 
In the remainder  of this paper we use as  invariants  the exponents $\ivr_i^{\ast}(S, \pp, \mathbf{a})$, which are
nonnegative integer-valued (or $+\infty$),
in place of  the $\pp$-sequence
ideals $\nu_i(S,\pp,\mathbf{a})$. 

%%%%%%%%%%%%%%%%%%%%%%%%%%%%%%%%%%%%%%%%%%%%%%%%%%%%%%%%%%%%
%The exponent sequence  is equivalent data to $\nu_i(S, \pp, \mathbf{a})$,
%because we have   $\ivr_i^{\ast}(S, \pp, \mathbf{a})$ determines $\nu_i(S, \pp, \mathbf{a})$ by \eqref{eq:nu-to-alpha}.
%Conversely, $\nu_i(S, \pp, \mathbf{a})$ determines $\ivr_i^{\ast}(S, \pp, \mathbf{a})$,
%under the convention that $\ivr_i^{\ast}(S, \pp, \mathbf{a})$ is the largest integer
%(or $+\infty$) making \eqref{eq:nu-to-alpha} valid. 
%%%%%%%%%%%%%%%%%%%%%%%%%%%%%%%%%%%%%%%%%%%%%%%%%%%%%%%%%%%% 

%%%%%%%%%
% Defn. 2.3
%%%%%%%%
\begin{defn}\label{defn:123}
 The \emph{$\pp$-exponent sequence of $\SSS$ associated to
 the $\pp$-ordering $\mathbf{a}$  } 
 is the sequence $\{ \ivr_i^{\ast} (S, \pp, \mathbf{a}):\,  i \ge 1\}$
 given by
 \begin{equation}\label{eq:def12}
\ivr_i^{\ast}(S,\pp,\mathbf{a}):=\order_\pp\left(\prod_{j=0}^{i-1}\left(a_i-a_j\right)\right).
\end{equation}
%%%%%%%%%%%%%%%%%%%%%%%%%%%%%
%$$
 %\nu_i(S, \pp, \mathbf{a}) = \pp^{\ivr_i(S, \pp, \mathbf{a})}.
% $$
%%%%%%%%%%%%%%%%%%%%%%%%%%%%%%%
 Here $\ivr_i^{\ast}(S, \pp, \mathbf{a})= \NN \cup \{+\infty\}$, with the convention to use $+\infty$ when $\nu(S, \pp, \mathbf{a}) = (0)$.
 %%%%%%%%%%%%%%%%%%%%%%%%%%%%%%%%%%%%%%%%%%%%%%
 %Convention might be used even when $\pp$ is a nilpotent prime ideal
% having $\pp^k= (0)$ for some finite $k$. (Bhargava's Dedekind rings may be finite rings.)
% In this paper we only consider integral domains.
%%%%%%%%%%%%%%%%%%%%%%%%%%%%%%%%%%%%%%%%%%%%%%
 \end{defn}

%%%%%%%%%%%%%%%%%%%%%%%%%%%%%%
%
% Subsection 2.2 Refined $p$-orderings
%
%%%%%%%%%% %%%%%%%%%%%%%%%%%%%
 \subsection{Bhargava's refined $\pp$-orderings }\label{subsec:22}

In \cite{Bhar:09}, Bhargava  introduced for Dedekind domains $\DD$ several refined versions of $\pp$-orderings,
which depend on extra parameters.

To define them, for a prime ideal $\pp$, let $\nu_{\pp}(\cdot): \DD \to \NN \bigcup \{+\infty\}$
be the additive valuation which   
for each nonzero  element $a \in \DD$  sets $\nu_{\pp}(a)=e$
where  $\pp^e$ is the highest power of $\pp$ dividing the principal ideal $(a)$,
and set  $\nu_{\pp}(0) = +\infty$.
%%%%%%%%%%%%%%%%%%%%%%%%%%%%%%%%%%%%%%%
% $\nu_{\pp}(\cdot)$ is an additive valuation on $\DD$.
%The refined orderings are expressed  in terms of this additive valuation.
%%%%%%%%%%%%%%%%%%%%%%%%%%%%%%%%%%%%%%%

%%%%%%%%%
% Defn. 2.4
%%%%%%%%
\begin{defn}[Refined $\pp$-orderings]\label{def:refined-p-ordering}
Given a set $\SSS$ of a Dedekind domain $\DD$ and a prime ideal $\pp$.
\begin{enumerate}
\item[(1)]
 An {\it $r$-removed $\pp$-ordering of $\SSS$},  $\mathbf{a}^{\{r\}} =(a_i^{\{r\}})_{i \ge 0}$, is one
 constructed as follows:

 %%%%%%%%%%%%%
% r-removed ( notation)
%(a_i^{\{r\}})_{i \ge 0})$)
%%%%%%%%%%%%%%
For fixed $r \in \NN$ choose $ a_i^{ \{r\}} \in \SSS$ for $0 \le i \le r$ arbitrarily. 
For larger values, having chosen $\{a_0^{\{r\}}, a_1^{\{r\}}, \cdots, a_{n-1}^{\{r\}}\}$,
   choose  at the next step for $a_n^{\{r\}}$ any $a_n \in \SSS$  with  a specified subset $A=A_{n,r}$ of 
   $\{0, 1, ..., n-1\}$ of cardinality $n-r$ that 
   together attain a  global minimum on the right side in:
\begin{equation}
\ivr_{n}^{\{r\}}( \SSS, \pp, \mathbf{a}^{\{r\}}) := \min_{a_n \in \SSS} \left( \min_{A} \sum_{i \in A} \iv_{\pp}(a_n- a_i^{\{r\}})\right),
\end{equation}
where $A$ ranges over all subsets $A \subset \{0, 1,\cdots, n-1\}$ of cardinality $n-r$.
We let $\sA_{n,r} := \{ a_j: j \in A_{n,r}\}$ denote the  $r$-removed 
set at index $n$ associated to the choice $a_n$. (The choice of $a_n$ and  $A_{n,r}$ may not be unique.)

\item[(2)]
A {\it $\pp$-ordering of order $h$ of $\SSS$},  $\mathbf{a}^{h}= (a_i^{h})_{i \ge 0}$,
is one constructed as follows: 
%%%%%%%%%%%%%
% order h (notation)
%(a_i^{h})_{i \ge 0})$)
%%%%%%%%%%%%%%
Choose $a_0 \in \SSS$ arbitrarily.
For a given $h \in \NN \cup \{ +\infty\}$, and having chosen $\{ a_0^{h}, a_1^{h}, ..., a_{n-1}^{h} \}$, 
choose  for $a_n^{h}$ any $a_n \in \SSS$ that  achieve a minimum on the right side in 
\begin{equation}
\ivr_{n}^{h}( \SSS, \pp, \mathbf{a}^{h}) := \min_{a_n \in \SSS} \left(\sum_{i=0}^{n-1} \min \left(h, \iv_{\pp}(a_n- a_i^{h})\right)\right).
\end{equation}

\item[(3)]
An {\it $r$-removed $\pp$-ordering of order $h$ of $\SSS$} , $\mathbf{a}^{h, \{r\}} = (a_i^{h,\{r\}})_{i \ge 0}$, is one constructed as
follows: 
 %%%%%%%%%%%%%%%%%%%
% r-removed, order h (notation)
%(a_i^{h, \{r\}})_{i \ge 0})$)
%%%%%%%%%%%%%%%%%%%
  For a given $r\in \NN$ and a given  $h \in \NN\cup\{+ \infty\} $, choose 
 $ a_i \in \SSS$ for $0 \le i \le r$ arbitrarily. Having chosen $\{ a_0^{h, \{r\}}, a_1^{h, \{r\}}, ..., a_{n-1}^{h, \{r\}} \}$, choose 
as $a_n^{h, \{r\}}$ any $a_n \in \SSS$ with  associated set $A=A_{n,r}^h$ of $n-r$ indices drawn from
$\{0, 1, \cdots n-1\}$  that together achieve a minimum on the right side in 
\begin{equation}
\ivr_{n}^{h, \{r\}}( \SSS, \pp, \mathbf{a}^{h, \{r\}}) := 
\min_{a_n \in \SSS} \left( \min_{A} \sum_{i \in A} \min \left(h, \iv_{\pp}(a_n- a_i^{h, \{r\}})\right) \right).
\end{equation}
We let $\sA_{n,r}^h$ denote the associated $r$-removed set  $\sA_{n,r}^h := \{ a_j: j \in A_{n,r}^h \}$
 of cardinality $n-r$,
which depends on the choice of $a_n=a_n^{h, \{r\}}$.
\end{enumerate} 
\end{defn}

The original notion of
$\pp$-ordering for a prime ideal $\pp$ corresponds to the parameter choice $r=0$ and $h=+\infty$.
The definition of constructing the refined $\pp$-orderings  simplifies in this case to
\begin{equation}
\ivr_k(S, \pp, \mathbf{a}) := \min_{a_k \in \SSS} \sum_{i=0}^{k-1} \ordr_{\pp} (a_k - a_i).
\end{equation} 
The additive valuation  product relation
\begin{equation}
\order_{\pp}(a_1 a_2) = \order_{\pp}(a_1) +\order_{\pp}(a_2)
\end{equation}
implies that 
\begin{equation}\label{eq:add-mult-prime}
\ivr_k(S, \pp, \mathbf{a}) = \sum_{i=0}^{k-1} \ordr_{\pp} (a_k - a_i) =
\ordr_{\pp} \left( \prod_{i=0}^{k-1} (a_k - a_i)\right)= \ivr_k^{\ast}(S, \pp, \mathbf{a}),
\end{equation} 
giving agreement with Definition \ref{def:p-ordering}.
Many of Bhargava's applications for $\pp$-orderings use the
 equality \eqref{eq:add-mult-prime} for prime ideals  in an essential way. 
 
 %%%%%%%%%%%%%%%%%%%%%%%%%%%%%%%%%%%%%%%%%%%%%%%%%%%%%
%\begin{equation}\label{eq:add-mult1}
%\sum_{i=0}^{n-1} \order_{\pp}\left(a_n-a_i\right) =\order_\pp\left(\prod_{i=0}^{n-1}\left(a_n-a_i\right)\right).
%\end{equation} 
%%which is valid for prime ideals $\pp$.
% (and which is in general invalid for general ideals $\bb$ that are not prime ideals).
%%%%%%%%%%%%%%%%%%%%%%%%%%%%%%%%%%%%%%%%%%%%%%%%%%%%%%

Bhargava \cite{Bhar:09} showed the following results about refined $\pp$-orderings, 
rephrased here in terms of $\pp$-exponent sequences.

%%%%%%%%%%
%
% Theorem 2.5
%
%%%%%%%%%%
\begin{thm}[Bhargava \cite{Bhar:09}]\label{thm:Bhar-main09}
Let $\DD$ be any   Dedekind domain, Let $\SSS$ be any subset of $\DD$, and  
and $\pp$ be any
nonzero prime ideal of $\DD$. 
Then the following hold.
\begin{enumerate}
\item[(1)]
An $r$-removed  $\pp$-ordering for $\SSS$ (with fixed  $r \in \NN$), denoted
${\bf a}^{\{r\}}= \{ a_k^{\{r\}} : k \ge 0 \}$, of   $\SSS$ yields
a well-defined $r$-removed  $\pp$-exponent sequence $\{ \alpha_k^{\{r\} }(S,  \pp, {\bf a} ):\, k \ge 0\}$, which
is independent of the $r$-removed $\pp$-ordering ${\bf a^{\{r\} }}$  of $\SSS$ chosen.
It gives  well-defined $\pp$-exponent sequence  invariants $\alpha_k^{\{r\} }(\SSS,  \pp)$ for each $k \ge 0$.
\item[(2)]
A $\pp$-ordering for $\SSS$ of order $h$ (for $h \in \NN \cup \{ +\infty\}$), denoted  ${\bf a}^{h}= \{ a_k^{h}: k \ge 0\}$ yields
a   $\pp$-exponent sequence of order $h$, $\{ \alpha_k^{h}(\SSS,  \pp, {\bf a} ): k \ge 0\}$, which 
is independent of the particular $\pp$-ordering ${\bf a}$ of  order $h$  of $\SSS$ chosen.
It gives  well-defined  $\pp$-exponent sequence invariants $\alpha_k^{h}(\SSS,  \pp)$ for each $k \ge 0$.
% chosen.
\item[(3)]
An  $r$-removed $\pp$-ordering of $\SSS$ of order $h$ (for $h \in \NN \cup \{ +\infty\}$ and $r \in \NN$)
denoted ${\bf a}^{h, \{r\}}= \{ a_k^{h, \{r\}} : k \ge 0\}$, 
has  a well-defined order $h$, $r$-removed  $\pp$-exponent sequence $\{ \alpha_k^{h, \{r\} }(S,  \pp, {\bf a} ): \,k \ge 0\}$, which
is independent of the particular  order $h$, $r$-removed $\pp$-ordering ${\bf a}^{h, \{r\} }$  of $\SSS$ chosen.
It gives well-defined $\pp$-exponent sequence invariants $\alpha_k^{h, \{r\} }(\SSS,  \pp)$ for each $k \ge 0$.

\end{enumerate}
\end{thm}

\begin{proof}
(1) is stated as Theorem 3 in \cite{Bhar:09}.
%%%%%%%%%%%%%%%%%%%%%%%%%%%%%%%
% and it is proved for local rings $R$ in Sect. 3.1.1 of \cite{Bhar:09}.
%%%%%%%%%%%%%%%%%%%%%%%%%%%%%%%

(2) is stated as Theorem 4 in \cite{Bhar:09}.
%%%%%%%%%%%%%%%%%%%%%%%%%%%%%%%
% and it is proved for local rings $R$ in Sect. 3.2.2 of \cite{Bhar:09}.
%%%%%%%%%%%%%%%%%%%%%%%%%%%%%%%
(3) is stated as Theorem 30 in \cite{Bhar:09}.
Bhargava  proves the cases (1) and (2)  for a local ring $\DD$ in \cite{Bhar:09}
and the general Dedekind domain case follows by localization 
at the given prime ideal $\pp$ 
of the given ring $\DD \to \DD_{\pp}=U^{-1}\DD$, 
(where $U= \DD\smallsetminus \pp$), which is an (injective) ring
homomorphism that  preserves the prime ideal property. 
 The paper  gives limited proof details for (3), 
 stating   (on page 29)  that Theorem 30  and the related 
Theorem 31  may be proved by combining all the techniques used in Sections 3 and 4.
\end{proof}

%%%%%%%%%%%%%%%%%%%%%%%%%%%%%%%%%%%%%%%%%%%%%%%%%%
%In the case $r=0$ and $h= +\infty$ of ``plain" $\pp$-orderings, these definitions recover the
%same $\pp$-exponent invariants as in Bhargava's original Definition \ref{def:p-ordering}. This 
%equality holds using the identity,
%valid for all $\SSS$-test sequences $\mathbf{a}$, 
%\begin{equation}\label{eq:add-mult-prime}
%\ivr_k(S, \pp, \mathbf{a}) = \sum_{i=0}^{k-1} \ordr_{\pp} (a_k - a_i) =
%\ordr_{\pp} \left( \prod_{i=0}^{k-1} (a_k - a_i)\right)= \ivr_k^{\ast}(S, \pp, \mathbf{a}).
%\end{equation} 
%Many of Bhargava's applications use the equality \eqref{eq:add-mult-prime} for prime ideals  in an essential way.  
%%%%%%%%%%%%%%%%%%%%%%%%%%%%%%%%%%%%%%%%%%%%%%%%%%%%%%

Bhargava used these refined  invariants to define 
 generalized factorials attached to the subset $S$ of  $\RRR$.
For the original Bhargava  \emph{factorial function for $\RRR$ associated to $S$, \, ($r=0, h= +\infty$)}, 
%%%%%%%%%%%%%%%%%%%%%%%%%%%%%
%here denoted\footnote{Bhargava \cite{Bhar:09} uses
%the notation $k!_{S, h}^{\{r\}}$.}
%%%%%%%%%%%%%%%%%%%%%%%%%%%%% 
$k!_{S}$, is defined 
to be  the $R$-ideal
\begin{equation}
k!_{S} :=\prod_{\pp} \nu_k (S,\pp) = \prod_{\pp} \pp^{\ivr_{k}^{\ast}(S, \pp)}.
\end{equation}
More generally, for  the factorial with extra order parameter $h$ and $r$-removed parameter $\{r\}$,
here denoted\footnote{Bhargava \cite{Bhar:09} uses
the notation $k!_{S, h}^{\{r\}}$.} 
$k!_{S}^{h, \{r\}}$, we have 
\begin{equation}\label{eq:product3}
k!_{S}^{h, \{r\}} :=\prod_{\pp} \nu_k^{h, \{r\}} (S,\pp) = \prod_{\pp} \pp^{\ivr_{k}^{k, \{r\}}(S, \pp)}.
\end{equation}

For $ \DD$ a Dedekind domain, all but finitely many terms in the product \eqref{eq:product3} are equal to the unit ideal $(1) =\DD$,
 i.e. almost all $\ivr_{k}(S, \pp)=0$ 
(\cite[Lemma 3]{Bhar:97a}). 
Thus for Dedekind domains   Bhargava's theory constructs factorials specified by  their (unique) prime factorizations.

Bhargava  \cite[Theorem 1]{Bhar:09} obtains the following consequence for generalized binomial coefficients.
%%%%%%%%%%%%%%%%%%%%%%%%%%%%%%%
%(For $\DD=\ZZ$ he showed it in  \cite[Theorem 8]{Bhar:00} )
%See also \cite{Bhar:97a}).
%%%%%%%%%%%%%%%%%%%%%%%%%%%%%%%%%
%%%%%%%%%%%%%%%%%

%%%%%%%%%%%
% Theorem 2.6
%%%%%%%%%%%
\begin{thm}\label{thm:binomials} 
For all Dedekind domains $\DD$, all subsets $\SSS$ of $\DD$, and all subsets $\sT$ 
of the set of all nonzero prime  ideals of $\DD$,  the  integral ideal $[k] !_{\SSS, \sT}^{h, \{r\}} [\ell] !_{\SSS, \sT}^{h, \{r\}}$
divides the integral ideal $[k+\ell] !_{\SSS, \sT}^{h, \{r\}}$ for all $k, \ell \ge0$.
\end{thm}

Bhargava proved this result for the set $\sT= \PP(\DD)$ of all nonzero prime ideals of 
a Dedekind domain $\DD$
as a special case of \cite[Theorem 1]{Bhar:09}. His  proof applies without change  allowing 
any subset $\sT$ of $\PP(\DD)$ in place of $\PP(\DD)$.
In terms of the $\pp$-exponent sequence invariants, the result follows from the inequality,
\begin{equation}\label{eqn:additive-invariant}
\ivr_{k + \ell}^{h, \{r\}}(\SSS, \pp) \ge \ivr_{k}^{h, \{r\}}(\SSS, \pp) + \ivr_{\ell}^{h, \{r\} }(\SSS, \pp),
\end{equation}
proved for each prime ideal $\pp$ separately.

%%%%%%%%%%
%%%%%%%%%%
%
% Section 3
%
%%%%%%%%%%
%%%%%%%%%%

\section{Main results}\label{sec:2}

The main results  extend Bhargava's theory of $\pp$-orderings for prime ideals $\pp$ in a Dedekind domain $\DD$
to treat $\bb$-orderings for all 
  ideals $\bb$ in  a Dedekind domain $\DD$.
 
 In what follows  $R$ denotes a general commutative ring with unit,  $\sI(R)$ denotes the set of all
 ideals of   $R$, and $\ddR$ denotes a field.
 %%%%%%%%%%%%%%%%%%%%%%%%%%%%%%%%
% ( The set $\sI^{\ast}(R)$ denotes all nonzero ideals of $R$.)
 %(A proper ideal $\bb$ is one with $\bb \ne \RRR$.)
% The notation    $\DD$ denotes a Dedekind domain, and $\ddR$ denotes a field, of  any characteristic.
%%%%%%%%%%%%%%%%%%%%%%%%%%%%%%%%% 
 The proofs treat $\bb$-orderings for all nonzero proper ideals, those in $\sI^{\ast\ast}(D)= D\smallsetminus \{ (0), D \}.$
 We treat separately   $\bb$-orderings for  the zero ideal $\bb=(0)$ and the full ring $\bb=(1)=\DD$.
 We use the convention that for the zero ideal 
\begin{equation}
\nu_k(S,  (0) ) = \begin{cases}  \DD  & \quad \mbox{ $k \le |\SSS|-1$} \\
(0) & \quad \mbox{otherwise}. 
\end{cases}
\end{equation} 

%%%%%%%%%%
%
% Subsection 3.1
%
%%%%%%%%%%

\subsection{Generalization of  $\bb$-orderings to arbitrary ideals $\bb$.}

We  first introduce general sequences of elements of a nonempty subset  $\SSS$ of 
a ring $R$, then specialize to $\bb$-orderings..

%%%%%%%%%%%%%%%%%%%%%%%%
%
% Definition  3.1 S-test sequences 
%
%%%%%%%%%%%%%%%%%%%%%%%%
\begin{definition}\label{def:test-seq}
Let $\SSS$ be a nonempty subset of a commutative ring $R$.
An {\em $\SSS$-test sequence }  $\mathbf{a}=\left(a_i\right)_{i=0}^\infty$ is any infinite 
sequence of elements  of $\SSS$, allowing repetitions of elements.
\end{definition} 

This notion 
%%%%%%%%%%%%%%%%%
%of {\em $\SSS$-test sequence} 
%%%%%%%%%%%%%%%%%
is different than the notion of {\em test set},
used in  Cahen and Chabert \cite{CC:18}, in connection with universal sets in the theory
of integer-valued polynomials, (see \cite{CC:97} for background.)

We define  ``plain" 
 $\bb$-orderings of $\SSS$ as a special subclass of $\SSS$-test sequences.
%%%%%%%%%%
%
% Definition 3.2
%
%%%%%%%%%%
\begin{definition}\label{def:b-orderings}
Let $R$ be a commutative ring with unit, and $\bb$ a nonzero  proper ideal of $R$.
 Let $S$ be a nonempty subset of  $R$. We call a $\SSS$-test sequence $\mathbf{a}=\left(a_i\right)_{i=0}^\infty$ 
 of elements of $S$ a  \emph{ 
 $\bb$-ordering of $\SSS$} if for all $i=1,2,3,\dots$,% $a_i$ satisfies
\begin{equation}\label{eqn:b-orderings}
\sum_{j=0}^{i-1}\order_{\bb}\left(a_i-a_j\right)=\min_{s\in \SSS}\,\sum_{j=0}^{i-1}\order_{\bb}\left(s-a_j\right),
\end{equation}
where $\order_{\bb}(a)$ is defined for $a \in R$ by
\begin{equation}\label{eqn:order-b}
\order_{\bb}(a):=\sup\left\{k\in\mathbb{N}: \,  aR \subseteq \bb^k \right\} \in \NN \cup \{ +\infty\}.
% \in \NN \cup \{ +\infty\}.
\end{equation}
For a Dedekind domain $\DD$ the condition $a \DD \subseteq \bb^k$ is
equivalent to $\bb^k$ divides $a\DD$.
\end{definition}

Given any initial value $a_0\in \SSS$, one can find a  $\bb$-ordering having that initial value, using the recurrence \eqref{eqn:b-orderings}
for later values. There will be more than one 
 $\bb$-ordering of $\SSS$, unless $\SSS$ is a singleton.

In Definition \ref{def:b-orderings} for the case $R$ is a Dedekind domain $\DD$ we note that
if $\bb=\pp$ is a prime ideal, then $\order_{\pp}$ is an additive valuation on $R$.
For prime ideals $\pp$ of $\DD$,  the terms on the right side  of \eqref{eqn:b-orderings}
in  Definition \ref{def:b-orderings} agree with the right side of \eqref{eqn:p-ordering0}  in Definition \ref{def:p-ordering}, due to the
valuation $\order_{\pp}$ identity \eqref{eq:add-mult-prime}. 
However when  $\bb$ is a composite ideal, $\order_{\bb}$ is {\em not} an additive valuation. 
It  satisfies instead that for all $a_1, a_2 \in \DD$, 
%$$
%\begin{equation}
$\order_{\bb}(a_1+ a_2) \ge \min( \order_{\bb}(a_1), \order_{\bb}(a_2))$
%\end{equation}
%$$
and 
\begin{equation}\label{eqn:non-valuation}
\order_{\bb}(a_1 a_2) \ge \order_{\bb}(a_1) +\order_{\bb}(a_2).
\end{equation}
where strict inequality may occur. 
(For example take $\DD=\ZZ$, $\bb=(6)$ and $a_1=2$, $a_2=3$; the right side is $0$ and the left side is $1$.)

The definitions of  refined $\bb$-orderings  
replace $\order_{\pp}$ with $\order_{\bb}$
in Definition \ref{def:refined-p-ordering}.

%%%%%%%%%
% Defn. 3.3
%%%%%%%%
\begin{defn}[Refined $\bb$-orderings]\label{def:refined-b-ordering}
Given a set $\SSS$ of a Dedekind domain $\DD$ and an arbitrary ideal $\bb$ of $\DD$.
\begin{enumerate}
\item[(1)]
 An {\it $r$-removed $\bb$-ordering of $\SSS$},  $\mathbf{a}^{\{r\}} =(a_i^{\{r\}})_{i \ge 0}$, is one
 constructed as follows:

 %%%%%%%%%%%%%
% r-removed
%(a_i^{\{r\}})_{i \ge 0})$)
%%%%%%%%%%%%%%
For fixed $r \in \NN$ choose $ a_i^{ \{r\}} \in \SSS$ for $0 \le i \le r$ arbitrarily. 
For larger values, having chosen $\{a_0^{\{r\}}, a_1^{\{r\}}, \cdots, a_{n-1}^{\{r\}}\}$,
   choose  at the next step for $a_n^{\{r\}}$ any $a_n \in \SSS$  with  a specified subset $A=A_{n,r}$ of 
   $\{0, 1, ..., n-1\}$ of cardinality $n-r$ that 
   together attain a  global minimum on the right side in:
\begin{equation}
\ivr_{n}^{\{r\}}( \SSS, \bb, \mathbf{a}^{\{r\}}) := \min_{a_n \in \SSS} \left( \min_{A} \sum_{i \in A} \iv_{\bb}(a_n- a_i^{\{r\}})\right),
\end{equation}
%over all $a_n \in S$ and 
where $A$ ranges over all subsets $A \subset \{0, 1,\cdots, n-1\}$ of cardinality $n-r$.
We let $\sA_{n,r} := \{ a_j: j \in A_{n,r}\}$ denote the  $r$-removed 
set at index $n$ associated to the choice $a_n$. (The choice of $a_n$ and  $A_{n,r}$ may not be unique.)

\item[(2)]
A {\it $\bb$-ordering of order $h$ of $\SSS$},  $\mathbf{a}^{h}= (a_i^{h})_{i \ge 0}$,
is one constructed as follows: 
%%%%%%%%%%%%%
% order h
%(a_i^{h})_{i \ge 0})$)
%%%%%%%%%%%%%%
Choose $a_0 \in \SSS$ arbitrarily.
For a given $h \in \NN \cup \{ +\infty\}$, and having chosen $\{ a_0^{h}, a_1^{h}, ..., a_{n-1}^{h} \}$, 
choose  for $a_n^{h}$ any $a_n \in \SSS$ that  achieve a minimum on the right side in 
\begin{equation}
\ivr_{n}^{h}( \SSS, \bb, \mathbf{a}^{h}) := \min_{a_n \in \SSS} \left(\sum_{i=0}^{n-1} \min \left(h, \iv_{\bb}(a_n- a_i^{h})\right)\right).
\end{equation}

\item[(3)]
An {\it $r$-removed $\bb$-ordering of order $h$ of $\SSS$} , $\mathbf{a}^{h, \{r\}} = (a_i^{h,\{r\}})_{i \ge 0}$, is one constructed as
follows: 
 %%%%%%%%%%%%%
% r-removed, order h 
%(a_i^{h, \{r\}})_{i \ge 0})$)
%%%%%%%%%%%%%%
  For a given $r\in \NN$ and a given  $h \in \NN\cup\{+ \infty\} $, choose 
 $ a_i \in \SSS$ for $0 \le i \le r$ arbitrarily. Having chosen $\{ a_0^{h, \{r\}}, a_1^{h, \{r\}}, ..., a_{n-1}^{h, \{r\}} \}$, choose 
as $a_n^{h, \{r\}}$ any $a_n \in \SSS$ with  associated set $A=A_{n,r}^h$ of $n-r$ indices drawn from
$\{0, 1, \cdots n-1\}$  that together achieve a minimum on the right side in 
\begin{equation}
\ivr_{n}^{h, \{r\}}( \SSS, \bb, \mathbf{a}^{h, \{r\}}) := 
\min_{a_n \in \SSS} \left( \min_{A} \sum_{i \in A} \min \left(h, \iv_{\bb}(a_n- a_i^{h, \{r\}})\right) \right),
\end{equation}
We let $\sA_{n,r}^h$ denote the associated $r$-removed set  $\sA_{n,r}^h := \{ a_j: j \in A_{n,r}^h \}$
% subset of $\{ a_0^{h, \{r\}}, a_1^{h, \{r\}}, \cdots, a_{n-1}^{h, \{r\}}\}$
 of cardinality $n-r$,
which depends on the choice of $a_n=a_n^{h, \{r\}}$.
\end{enumerate} 
\end{defn}
The case $r=0$ and $h = +\infty$ coincides with  the notion of
$\bb$-ordering  in Definition \ref{def:b-orderings}  for arbitrary ideals $\bb$,
 which, in this context, we 
call {\em ``plain" $\bb$-ordering}.

To each $\bb$-test sequence of $S$, we can associate $\bb$-exponent sequences 
corresponding to these measures  of refined $\bb$-orderings.

%%%%%%%%%
%
% Definition 3.4
%
%%%%%%%%%

\begin{defn}[Refined $\bb$-exponent sequences for arbitrary $\SSS$-test sequences]\label{def:associated}
Let $\SSS$ be a nonempty subset of a commutative ring $R$,
 let $\bb$ be a nonzero proper ideal in $R$, and let $\mathbf{a}=(a_n)_{n=0}^{\infty}$
be any fixed  $\SSS$-test sequence. 
\begin{enumerate}
\item[(1)]
%({\it $r$-removed $\bb$-eof $\SSS$} 
%If $\mathbf{a}^{\{r\}} = (a_i^{\{r\}})_{i \ge 0}$ is an $r$-removed $\bb$-ordering for $\SSS$,
Given $\mathbf{a}$
% for each we choose an associated   minimizing sequence $( A_{n, r})_{n=0}^{\infty}$, in which 
% for each $n$ the each
 for fixed $n$ 
 % $A= A_{n,r}$  minimizes
%$\sum_{i \in A} \iv_{\bb}(a_n- a_i)$ over all $A \subset \{0,1, \cdots , n-1\}$ of cardinality $n-r$.
%with $|A_{n-r}|= \max (n-r,0)$,
the   {\em  $r$-removed $\bb$-exponent sequence of ${\bf{a}}$}  is 
\begin{equation}
\ivr_n^{\{r\}}(\SSS, \bb, \mathbf{a}) := \min_{A} \sum_{i \in A} \iv_{\bb}(a_n - a_i),
\end{equation}
where the minimum is taken over  all $A \subset \{0,1, \cdots , n-1\}$ of cardinality $n-r$.
We let $A=A_{n,r}$ be a choice of minimizer,
and  then  let $\sA_{n,r}$ denote the associated minimizing  subset of $\{a_0, a_1, \cdots, a_{n-1}\}$ of cardinality $\max(n-r,0)$.
The choice of  $A_{n,r}$ need not be unique. However  
the values for each fixed $n$ of  $\ivr_n^{\{r\}}(\SSS, \bb, \mathbf{a})$ are well-defined.
We obtain a 
%(possibly non-unique) 
associated minimizing sequence $( A_{n, r})_{n=0}^{\infty}$.
\item[(2)]
%({\it $\bb$-ordering of order $h$ of $\SSS$}
%For  ${\mathbf a}$ 
%$\mathbf{a}^{h} = (a_i)^{h}_{i \ge 0}$ is an order $h$ $\bb$-ordering for $\SSS$ then
The {\em  $\bb$-exponent sequence of order $h$ of ${\mathbf a}$ }
  is 
\begin{equation}
\ivr_n^{h}(\SSS, \bb, \mathbf{a}) := \sum_{i=0}^{n-1} \min \left(h, \iv_{\bb}(a_n- a_i)\right).
\end{equation}
\item[(3)]
%If $\mathbf{a}^{h, \{r\}} = (a_i^{h, \{r\}})_{i \ge 0})$
 %is an $r$-removed $\bb$ sequence of order $h$, with s
 Given $\mathbf{a}$ the
 %with  associated  $r$-minimizing sets $( A_{n, r})_{n=0}^{\infty}$ as in (1), the
  {\em  $r$-removed $\bb$-exponent sequence of order $h$ of ${\mathbf a}$}  is
\begin{equation}
\ivr_n^{h, \{r\}}(\SSS, \bb, \mathbf{a}) :=\min_{A} \sum_{i \in A} \min \left(h, \iv_{\bb}(a_n- a_i)\right),
\end{equation}
where the minimum is taken over all $A \subset \{0,1, \cdots , n-1\}$ of cardinality $n-r$.
and  we then let $\sA_{n,r}^h$ denote the associated minimizing subset of $\{a_0, a_1, \cdots, a_{n-1}\}$ of cardinality $n-r$.
We obtain an associated minimizing sequence $( A_{n, r}^h)_{n=0}^{\infty}$.
\end{enumerate} 
\end{defn}

%%%%%%%%%%
%
% Definition 3.5
%
%%%%%%%%%%
\begin{definition}[$\bb$ and $\sI$-Sequenceability]\label{def:23}
%(1)
 Given a  commutative ring $R$ and a   nonzero proper
ideal $\bb$, we say that $R$ is  {\em  $\bb$-sequenceable} 
(for a given type of refined $\bb$-ordering) if for every  nonempty subset $S \subset R$, any
two 
$\bb$-orderings $\mathbf{a}_1^{h, \{r\}}$, $\mathbf{a}_2^{h, \{r\}}$ for $\SSS$
have identical associated $\bb$-exponent sequences:  
$\ivr_i(S,\bb,\mathbf{a}_1^{h, \{r\}})= \ivr_i(S,\bb,\mathbf{a}_2^{h, \{r\}})$
for all $i \ge 0$. If so, we  set
\begin{equation}\label{def:b-sequence-ind}
 \ivr_i^{h, \{r\}}(S, \bb) := \ivr_i^{h, \{r\}}(S, \bb, \mathbf{a}) 
 \end{equation}
 for any 
 $\bb$-ordering
 $\mathbf{a}$ and term it 
the {\em  $\bb$-exponent sequence of $\SSS$} for the ring $R$.

(2) We say that the ring  $R$ is  {\em $\II$-sequenceable}
(for a given type of refined $\bb$-ordering)  if it is  $\bb$-sequenceable for all 
nonzero proper ideals $\bb$ of $R$. 
 \end{definition}

The  main result of this paper is  the
well-definedness of Bhargava's refined ordering invariants ($r$-removed of order $h$)
for all ideals $\bb$ of  a Dedekind domain $\DD$, including $\bb =(0)$ and $\bb=(1)$.
 
%%%%%%%%%%
%
% Theorem 3.6 NEW MAIN THM
%
%%%%%%%%%%
\begin{thm}[Well-definedness of the refined $\bb$-exponent sequences of $S$ in Dedekind domains]\label{thm:NEW-MAIN}
Let $\DD$ be any   Dedekind domain, let $\SSS$ be any subset of $\DD$, and  
and $\bb$ be  any
ideal of $\DD$. 
Then the following hold.
\begin{enumerate}
\item[(1)]
An $r$-removed  $\bb$-ordering for $\SSS$ ( with fixed  $r \in \NN$), denoted
${\bf a}^{\{r\}}= \{ a_k^{\{r\}} : k \ge 0 \}$, of   $\SSS$ has
an  $r$-removed  $\bb$-exponent sequence $\{ \alpha_k^{\{r\} }(S,  \bb, {\bf a} ):\, k \ge 0\}$, which
is independent of the $r$-removed $\bb$-ordering ${\bf a^{\{r\} }}$  
%of $\SSS$ 
chosen.
It gives  well-defined $\bb$-exponent sequence invariants $\alpha_k^{ \{r\}}(\SSS, \bb)$ for each $k \ge 0$.
\item[(2)]
A $\bb$-ordering for $\SSS$ of order $h$ (for $h \in \NN \cup \{ +\infty\}$), 
denoted  ${\bf a}^{h}= \{ a_k^{h}: k \ge 0\}$, 
has an 
%well-defined 
order $h$ $\bb$-exponent sequence $\{ \alpha_k^{h}(S,  \bb, {\bf a} ): k \ge 0\}$, which
is independent of the particular $\bb$-ordering ${\bf a}^{h}$ of  order $h$  
%of $\SSS$
 chosen.
It gives well-defined $\bb$-exponent sequence invariants $\alpha_k^{(h)}(S, \bb)$ for each $k \ge 0$.
% chosen.
\item[(3)]
An  $r$-removed $\bb$-ordering of $\SSS$ of order $h$ (for $h \in \NN \cup \{ +\infty\}$ and $r \in \NN$),
denoted ${\bf a}^{h, \{r\}}= \{ a_k^{h,\{r\}} : k \ge 0\}$,
has an 
%well-defined 
order $h$, $r$-removed  $\bb$-exponent sequence $\{ \alpha_k^{h, \{r\} }(S,  \bb, {\bf a} ): \,k \ge 0\}$, which
is independent of the particular  order $h$, $r$-removed $\bb$-ordering ${\bf a}^{h, \{r\} }$   chosen.
It gives  well-defined $\bb$-exponent sequence  invariants $\alpha_k^{h, \{r\}}(\SSS, \bb)$ for each $k \ge 0$.
\end{enumerate}
\end{thm}

 The   specialization of these $\bb$-exponent sequences to the case of  prime ideals $\pp$ agrees with those in 
% Bhargava's definition of $\pp$-exponent sequences 
 Bhargava's  Definition \ref{def:refined-p-ordering}.
 We state the assertions (1)-(3) separately because they rely on well-definedness  results of Bhargava for
 generalized  $\pp$-orderings 
 in local rings given
 in Theorem \ref{thm:T-invariant}.
 Theorem \ref{thm:well-definedness0} for ``plain" $\bb$-orderings is the special case $r=0$ of Theorem \ref{thm:NEW-MAIN} (1); it has
 a  direct proof  independent of Bhargava's results via Theorem \ref{thm:max-min}.

%Note that $\ivr_k(S, \bb, \mathbf{a}) \ge 0$ for all ideals $\bb$ and all $S$-test sequences $\mathbf{a}$. 

We prove  Theorem 
  \ref{thm:NEW-MAIN} in Section \ref{sec:5}.
 % especially Section \ref{subsec:proofs-53}.
 %This result is proved assuming the  truth of a special case of
 The main content of the  proof is  a mechanism giving, for all nonzero proper ideals,  a reduction to a   case covered by
Bhargava's theorem \ref{thm:Bhar-main09},  
which is $\pp$-sequenceability
 for  the prime ideal $\pp = t \KKK[[t]]$ of the
 formal power-series ring over a field $\KKK[[t]]$,
 for the different generalized $\pp$-sequences,
 stated as Theorem \ref{thm:T-invariant}. We apply (Bhargava's) Theorem \ref{thm:T-invariant} as a black box.
 The remaining cases $\bb= (0)$ and $\bb= (1)$ covered by Theorem \ref{thm:NEW-MAIN} are given separate proofs.
 This reduction proceeds by  embedding 
 a  given Dedekind domain $\DD$ into a larger Dedekind domain $\DD^{'}$,
 and giving an injective  map  $\varphi_{\bb} (\cdot)$ of $\DD$ into the unique factorization domain $\DD^{'}[[t]]$,
 which need not be a Dedekind domain, but is  a subring 
 %The  ring $\DD'[[t]]$ need not be  a Dedekind domain, but sits inside $\DD'[[t]]
 of  $\KKK^{'}[[t]]$, where $\KKK^{'}$ is the quotient field of $\DD'$, which
 is a Dedekind domain.
 The embedding maps  $\varphi_{\bb}(\cdot)$   have the feature   
 that they are not required to  preserve either of the ring operations
 (ring addition and ring multiplication),  but the sequenceability 
%(invariant) 
property  is transferred.  
 %(Bhargava's proofs use ring homomorphisms.)

%%%%%%%%%%%
%
% Subsection 3.2
%
%%%%%%%%%%

\subsection{Generalized factorial ideals}

Using  $\bb$-sequences, we define generalizations  of Bhargava's notion
of factorial ideals  
for Dedekind domains, as treated  in  \cite{Bhar:98}, \cite{Bhar:00}.

%%%%%%%%%%
%
% Definition 3.7
%
%%%%%%%%%%
\begin{definition}\label{definition:factorials-S-T}
 Let $S$ be a nonempty subset of the  $\II$-sequenceable  ring $\RRR$. Let $\T\subseteq\sI(\DD)$. 
For $k=0,1,2,\dots$, the \emph{$k$-th generalized factorial  ideal 
 associated to $S$ and $\T$}, denoted $[k]!_{S,\T}^{h, \{r\}}$, is defined by
\begin{equation}\label{eqn:factorials-S-T}
[k]!_{S,\T}^{h, \{r\}}:=\prod_{\bb \in\T}\bb^{\ivr_k^{h, \{r\}}(S,\bb)}.
\end{equation}
\end{definition}

We note that:
\begin{enumerate}
\item[(1)] If $\T=\varnothing$, then the product on the right side of \eqref{def:factorials-S-T} is empty; so $[k]!_{S,\varnothing}^{h, \{r\}}=1$.
\item[(2)] If $\T\neq\varnothing$, then $[k]!_{S,\T}^{h, \{r\}}$ is a nonzero ideal
%(finite) positive integer
 if and only if $k<\vert S\vert- r$.
\item[(3)]
% From   \eqref{def:b-sequence-ind}, 
We have  $\ivr_0^{h, \{r\}}(S,\bb)=0$
 for all $\bb\in\sI(\DD)$; so 
$[0]!_{S,\T}^{h, \{r\}}=(1)$ is the unit ideal.
\item[(4)] The unit ideal $\bb= (1) \in \T$ plays no role in generalized factorial ideals, as its associated factors in \eqref{eqn:factorials-S-T} are always $(1)$,
for all $k \ge 0$.
\item[(5)] The 
zero ideal $\bb=(0) \in \T$ plays no role when $\SSS$ is an infinite set.
In  that case for $k\ge 0$  all its associated  factors in \eqref{eqn:factorials-S-T} are $(1)$. ( All $\ivr_k^{h, \{r\}}(S,(0))=0$ and  $(0)^{0}=(1)$.)
It does play  a role when $\vert \SSS\vert$ is finite.
%In that case for $k \ge (r+1)\vert \SSS \vert$, one has $\ivr_k^{h, \{r\}}(S,(0)) \ge 1$
In that case  all generalized factorials
become $(0)$ for  $k \ge (r+1)\vert \SSS \vert$,  for  any $h \ge 1$, see the proof of Theorem \ref{thm:NEW-MAIN}. 
 (We have $\ivr_k^{h, \{r\}}(S,(0))\ge 1$ for $k \ge (r+1)\vert \SSS \vert$ because duplicate elements occur.)
\end{enumerate}

We show 
 that generalized factorials are well-defined by their product formula \eqref{eqn:factorials-S-T},
 see in Lemma \ref{lem:60}.
 We may restrict $\T$ to the set of  proper ideals, and we also require  $k$  to satisfy $k < (r+1)\vert \SSS \vert$
 if  $(0) \in \T$.

It is known for plain $\pp$-orderings that $[1]!_{S, \T}=(1)$ is the unit ideal as long as $S$ is not contained in a single congruence class modulo $\bb$ for every $\bb\in\T$.
It need not always be the unit ideal.
%%%%%%%%%%%%%%%%%%%%%%%%%%%%%%%%%%%%%%%%%%%%%%%%%%%%%%%%%%
%A sufficient condition for this is that the  set of all $s_i-s_j$ ($s_i, s_j \in S$)  generates the unit ideal $(1)$ in $\DD$.
%%%%%%%%%%%%%%%%%%%%%%%%%%%%%%%%%%%%%%%%%%%%%%%%%%%%%%%%%%

%%%%%%%%%%
%
% Example 3.8
%
%%%%%%%%%%
\begin{example}\label{exam:39}
For  Dedekind domains  $\DD$,   the choice  $\T=\PP_{\DD} :=\Spec(\DD) \smallsetminus \{ (0)\}$  
recovers  Bhargava's generalized factorial ideals \cite[Section 2]{Bhar:98}. 
Further specialization to $\DD= \ZZ$ contains the usual factorial function as the special case $(S,\T)=(\mathbb{Z},\PP)$:
$$
[k]!_{\mathbb{Z},\PP}=(k!), 
$$
on identifying the factorial ideal $(k!)$ with its positive generator $k! \in \NN$ (\cite{Bhar:00}).
\end{example}

Generalized factorial ideals satisfy inclusion relations in the index $\sT$.
%%%%%%%%%%%%%%%%%%%%%%%%%%%%%%%%
%
% Proposition 3.9: Factorial inclusion and divisibility properties
%
%%%%%%%%%%%%%%%%%%%%%%%%%%%%%%%
\begin{prop}[Factorial Inclusion and divisibility properties]\label{prop:factorial-ordering}
Let $S$ be a nonempty subset of the Dedekind domain $\DD$. Let $\T_1\subseteq\T_2\subseteq\sI (\DD)$. Then for integers $0\le k<\vert S\vert$,
$$
[k]!_{S,\T_1}^{h, \{r\}}\quad\text{divides}\quad[k]!_{S,\T_2}^{h, \{r\}}.
$$
Equivalently
$$
[k]!_{S,\T_2}^{h, \{r\}} \quad\text{is contained in }\quad[k]!_{S,\T_1}^{h, \{r\}}.
$$
\end{prop}

We prove Proposition \ref{prop:factorial-ordering} in Section \ref{subsec:61}.
We discuss  results  inclusion relations when  varying the subset $S$, holding $\T$
fixed in Section \ref{sec:8}.

%%%%%%%%%%
%
% Subsection 3.3
%
%%%%%%%%%%

\subsection{Generalized-integer ideals and generalized binomial coefficient ideals}\label{subsec:33}

Integral  domains $\ddS$ have a notion of {\em fractional ideal} $\cc$, defined to be  
the subset of elements  $\cc$
of its quotient field $K$, such that there is a nonzero element $b \in \ddR$ for which
$\frac{1}{b} \cc$ is an (integral) ideal of $\ddS$.
The ideal multiplication on $\ddS$ induces  a well-defined multiplication on fractional ideals so that 
they form a commutative monoid with $\ddS$ as the identity
element $(1)$. A fractional ideal $\cc$ is  {\em invertible} if and only if there is a fractional
ideal $\dd$ such that $\cc \dd = (1)$. Dedekind domains $\DD$ are characterized
by  the property that every nonzero
fractional ideal of $\DD$ is invertible. (\cite[Chap 13, Basic Property 4.2(c)]{AusB:74}).

%%%%%%%%%%
%
% Definition 3.10
%
%%%%%%%%%%
\begin{definition}\label{definition:integers-S-T}
Let $S$ be a nonempty subset of a Dedekind domain $\DD$.
Let $\T \subseteq \sI(\DD)$. For positive integers $n<\vert S\vert$, 
the \emph{$n$-th  generalized-integer  ideal associated to $S$ and $\T$}, denoted $[n]_{S,\T}$, is 
the fractional ideal of $\DD$ defined by
\begin{equation}\label{def:nST}
[n]_{S,\T}^{h, \{r\}}:=\frac{[n]!_{S,\T}^{h, \{r\}}}{[n-1]!_{S,\T}^{h, \{r\}}}.
\end{equation}
\end{definition}

We set $[0]_{S, \T} ^{h, \{r\}}= (1)$ the unit ideal, whence  $[1]_{S,\T}^{h, \{r\}}= [1] !_{S,\T}^{h, \{r\}}.$
 The definition \eqref{def:nST} represents a generalized factorial ideal as a product of generalized-integer ideals: 
\begin{equation}\label{eqn:gen-integer-factorial-product}
[n]!_{S,\T}^{h, \{r\}} := \prod_{j=1}^n [j]_{S,\T}^{h, \{r\}},
\end{equation} 
which follows by induction on $n  \ge 2$.

This  definition \eqref{def:nST}  defines  $[n]_{S, \T}$  as  a fractional ideal in $\DD$.
The next result asserts that $[n]_{S, \T}$  is an integral ideal (justifying the name ``generalized-integer").

%%%%%%%%%%%
%%
%% Theorem 3.11
%%
%%%%%%%%%%%

\begin{thm}\label{thm:integer-integer}
Let $\SSS$ be a  nonempty subset of 
%an $\sI$-sequenceable integral domain $\ddR$. 
a Dedekind domain $\DD$. 
Let $\T\subseteq\sI(\DD)$. Then for positive integers $n<\vert S\vert$, the generalized-integer ideal  $[n]_{S,\T}^{h, \{r\}}$ 
is an integral ideal of $\DD$.
\end{thm}

We prove Theorem \ref{thm:integer-integer} in Section \ref{subsec:61}.

%%%%%%%%%%
%
% definitionn 3.12
%
%%%%%%%%%%
\begin{definition}\label{definition:binomial-S-T}
Let $S$ be a nonempty subset of a Dedekind domain $\DD$. Let $\T\subseteq \sI(\DD)$. For  integers $k, \ell \ge0$
with 
$ k +\ell<\vert S\vert$, the \emph{generalized binomial coefficient  associated to $S$ and $\T$}, denoted ${k+\ell \brack\ell}_{S,\sT}$, is defined by
\begin{equation}
{k+\ell \brack\ell}_{S,\sT}^{h, \{r\}}:=\frac{[k+\ell]!_{S,\T}^{h, \{r\}}}{[k!]!_{S,\sT}^{h, \{r\}}[\ell]!_{S,\sT}^{h, \{r\}}}.
\end{equation}
\end{definition}

By definition the set ${k+\ell \brack\ell}_{S,\T}$ is a fractional ideal. The next result asserts it is an integral ideal.
%%%%%%%%%%
%
% Theorem 3.13
%
%%%%%%%%%%
\begin{thm}\label{thm:binomial-integer}
Let $S$ be a nonempty subset of a Dedekind domain $\DD$ . Let $\T\subseteq \sI(\DD)$. Then for integers 
$k, \ell \ge 0$ with  $ k+\ell<\vert S\vert$, the generalized binomial coefficient ${k+\ell\brack\ell}_{S,\T}^{h, \{r\}}$ is an integral ideal. That is, as integral ideals, 
$[k]!_{S,\sT}^{h, \{r\}}[\ell]!_{S,\sT}^{h, \{r\}}$ divides $[k+ \ell]!_{S,\T}^{h, \{r\}}$.
\end{thm}

We prove Theorem \ref{thm:binomial-integer} in Section \ref{subsec:61}.

%%%%%%%%%%%
%
%
%  3.4 Contents of paper
%
%
%%%%%%%%%%%%
\subsection{Contents}
The contents of the remainder of the  paper are as follows.
\begin{enumerate}
 \item
 %Section  4
 Section \ref{sec:Bhargava-K[[t]]} studies  $t$-orderings and $t$-exponent sequences for
 the prime ideal $\pp=(t)\ddR[[t]]$ of  a formal power series ring
 $\ddR[[t]]$, where $\ddR$ is a field (of any characteristic). It states results
 of Bhargava on  the well-definability of
 $t$-sequences of different sorts  for  $(t)\ddR[[t]]$ (Theorem \ref{thm:T-invariant}).
 It shows  a local inequality relating the $k$-th, $\ell$-th and $(k+ \ell)$-th $\pp$-exponent invariants 
 for a fixed $S$.(Theorem \ref{thm:bhar-integral}). 
  It proves a nondecreasing property for
 such sequences as $k$ varies directly from the definition (Proposition \ref{prop:t-sequence-increasingA}). 
 \item
 % Section 5
  Section \ref{sec:propertyC} formulates  a  Property  C(R, $\bb, \ddR)$
  (Property C for short), where $\bb$ is an ideal of $R$ and $\ddR$ is a field.  It  shows 
  that when a map $\phi_{\bb}: R \to \ddR[[t]]$ between rings exists that  has 
 Property C(R, $\bb$, $\ddR$),  then  the existence 
 and well-definedness of $\bb$-invariants for all subsets $S$ 
 of $R$ holds. This map $\phi_{\bb}$ need not preserve either  ring operation:
 ring   addition or ring multiplication.
 %A  ring $R$ called {\em $\sI$-sequenceable} if such Property C(R, $\bb$, K)  maps exist
 %for all its nonzero ideals $\bb$.
 \item
 %Section 6
 Section \ref{sec:5}  
   establishes that Dedekind domains $\DD$ are $\sI$-sequenceable
   by showing  that  Property C($\DD, \bb, \ddR$) holds
 for all  nonzero ideals $\bb$ of all Dedekind domains $\DD$,
 for a suitable field $K$.
 %that is, they 
 %are $\bb$-sequenceable for all nonzero ideals $\bb$. 
 We use the fact that an arbitrary Dedekind
 domain $\DD$ can be injectively imbedded into a larger Dedekind domain $\DD'$ that is a principal ideal domain.
 This is a result of Claborn \cite{Clab:65a}.
 We construct for a given nonzero ideal $\bb$ or $\RRR$  a map $ \phi_{\bb}: \RRR \to \ddR'[[t]]$ 
 where $\ddR'$ is the quotient field of $\DD'$, satisfying property $C(S, \bb, \ddR')$.
 In Section \ref{subsec:proofs-general} we 
 complete the proof of the well-definedness theorem \ref{thm:NEW-MAIN}. 

 \item
 % Section 7
 Section \ref{sec:6} gives proofs of  results of Section \ref{sec:2} about generalized factorial ideals
 and generalized binomial coefficients for rings $\RRR$ that are $\sI$-sequenceable.

 \item
 %Section 8
 Section \ref{sec:8}  discusses extensions of this work.  
 %\item
 %Appendix
 %An appendix gives tables of generalized factorials 
% and generalized binomial coefficients for the case $\DD= \ZZ$
% and $S= \ZZ$. 
 \end{enumerate}
%%%%%%%%%%
%%%%%%%%%%
%
% Section 4
%
%%%%%%%%%%
%%%%%%%%%%

\section{$\tp$-orderings and $\tp$-sequences  in the power series ring $\ddR[[t]]$}\label{sec:Bhargava-K[[t]]}

%%%%%%%%%%%%%%%%%%%%%%%%%%%%%%%%%%%%%%%%%%%%%%%%%%%%%%%%%%%%%%%%%%%%%%%%%%%%
Throughout this section, we let $\ddR$ be a field of any characteristic.
%we let $\KK$ be a field, and 
We consider the
formal power series ring $R = \ddR[[t]]$,  
together with its  unique nonzero prime ideal
$(t) = t \ddR[[t]]$. The ring $R$ is a discrete valuation ring, so is a Dedekind domain.
We treat the special case of   Bhargava's theorem of the well-definedness
of his refined $\pp$-sequences for any subset $U$ of this ring, with $\pp= (t)$;
we term these $t$-sequences, stated as Theorem \ref{thm:T-invariant}. 
We derive  monotonicity inequalities for the invariants. 
 
%%%%%%%%%%%%%%%%%%%%%%%%%%%%%%
%
% Subsection 4.1
%
%%%%%%%%%% %%%%%%%%%%%%%%%%%%%%
 \subsection{Bhargava $t$-sequence invariants for local rings $K[[t]]$}\label{subsec:41}

Bhargava's proofs in \cite{Bhar:09} of the well-definedness of the associated
$\pp$-sequences proceed by reducing to the local ring case.
For our results it suffices to  know 
that his results hold  for formal  power series rings $\DD = \KKK[[t]]$
over a field $\KKK$ of any characteristic, for the single prime ideal  $\pp= t \KKK[[t]]$; we rename
these $\pp$-orderings and associated $\pp$-exponent sequences
``$t$-orderings " and ``$t$-exponent sequences", respectively.

Let $K$ be a field of any characteristic, and let $\ddR[[t]]$ be the ring of formal power series over $\ddR$. 
We use an additive valuation $\dnu:\ddR[[t]]\rightarrow\{0,1,2,\dots\}\cup\{+\infty\}$ for nonzero $f(t)=\sum_{i=0}^\infty\dda_it^i$ by
$$
\dnu(f(t))=\min\left\{i\in\mathbb{N}:\dda_i\neq0\right\} \in \NN,
$$
and we set $\dnu(0)=+\infty$.
That is,
$$
\dnu(f(t))=\sup\left\{\alpha\in\mathbb{N}:t^\alpha\text{ divides }f(t)\right\}.
$$

%%%%%%
% Prop. .1 [OMITTED AS PROPOSITION FOR NOW] [STATED IN TEXT; NO PROOF]
%%%%%%%
%\begin{prop}\label{prop:valuation}
The function $\dnu$ is an (additive) discrete valuation on $\ddR[[t]]$. That is, $\dnu: \ddR \to \NN \cup \{ +\infty\}.$
\begin{equation}\label{add-nu}
\dnu(f(t)\pm g(t))\ge\min\left\{\dnu(f(t)),\dnu(g(t))\right\},
\end{equation}
and 
\begin{equation}\label{mul-nu}
\dnu(f(t)g(t))=\dnu(f(t))+\dnu(g(t)).
\end{equation}
%\end{prop}

%Bhargava derived a consequence of this result.

%%%%%%%%%%%%%%%%%%%%%%%%%
%%%%%%%%%%%%%%%%%%%%%%%%
%  Theorem 4.1 [Formal power series ring]
%
%%%%%%%%%%%%%%%%%%%%%%%%
\begin{thm}[Bhargava] \label{thm:T-invariant}
For  the formal power series ring $\ddR[[t]]$, for all nonempty sets  $\ddS \subseteq \ddR[[t]]$,
%which contain at least two distinct elements,
and any $\tp$-ordering $\mathbf{f}=\left(f_i(t)\right)_{i=0}^\infty$ of $\ddS$ 
for the prime ideal $\tp= t \ddR[[t]]$ of the following four types
(plain, $r$-removed, order $h$, or $r$-removed of order $h$),  
 the corresponding $t$-exponent sequences 
 %  sequence of nonnegative integers $\left(\ddnu_k(\ddS,\mathbf{f})\right)_{k=0}^\infty$ (allowing $+\infty$
%as a value) defined by
\begin{eqnarray}\label{t-sequence-definition}
\ddnu_k( \ddS,   t \ddR[[t]], \mathbf{f}) &:=& \sum_{j=0}^{k-1}\dnu\left(f_k(t)-f_j(t)\right) \in \NN \cup \{ +\infty\} \label{eq:43}\\
\ddnu_k^{\{r\}} ( \ddS,t \ddR[[t]],\mathbf{f}) &:=&\min_{\sA_{k,r}}\left[ \sum_{j \in \sA_{k,r}}\dnu\left(f_k^{\{r\}}(t)-f_j^{\{r\}}(t)\right)\right] \in \NN \cup \{ +\infty\} \label{eq:44}\\
\ddnu_k^{h}(\ddS,t \ddR[[t]], \mathbf{f}) &:=& \sum_{j=0}^{k-1}\min[h, \dnu\left(f_k^{h}(t)-f_j^{h}(t)\right)] \in \NN \cup \{ +\infty\} \label{eq:45}\\
\ddnu_k^{h, \{r\}}( \ddS,t \ddR[[t]], \mathbf{f}) &:=& \min_{\sA_{k,r}}\left[\sum_{j \in \sA_{k,r}}
\min[h,\dnu\left(f_k^{h,\{r\}}(t)-f_j^{h, \{r\}}(t)\right)]  \right]\in \NN \cup \{+\infty\} \label{eq:46}
\end{eqnarray}
are each  independent of the choice of $\tp$-ordering $\mathbf{f}$ of $\ddS$ of the same type.
(For $r$-removed sequences,  $\sA_{k,r}$ runs over all subsets of $\{0,1,..., k-1\}$ of cardinality $k-r$). 
% $(t)$-ordering $\mathbf{f}=\left(f_i(t)\right)_{i=0}^\infty$ of $\ddS$.
\end{thm}

\begin{proof}
(1) is proved for local rings $R$ in Sect. 3.1.1 of \cite{Bhar:09}.
(2) is  proved for local rings $R$ in Sect. 3.2.2 of \cite{Bhar:09}.
(3) is a special case of  Theorem 30 stated in \cite{Bhar:09}. 
Proof details for Theorem 30 are   sketched on  p. 29 of \cite{Bhar:09}.
\end{proof} 

The case of ``plain" $t$-exponent sequences occurs when $r=0$ and $h= +\infty$. 
In \cite{LY:24a} we   gave an independent proof of this special case \eqref{eq:43} of
Theorem \ref{thm:T-invariant}, using the max-min criterion stated in Theorem \ref{thm:max-min}.
 %in the Appendix. 

%Bhargava derived a consequence of this result.
%%%%%%%%%%%
% Theorem 4.2
%%%%%%%%%%%
\begin{thm}[Bhargava]\label{thm:bhar-integral}
For all nonempty subsets $U$ of $\KKK[[t]]$, and the ideal $(t):= t \KKK[[t]]$,
and all integers $k, \ell \ge 0$, one has 
\begin{equation}
 \ivr_{k+\ell}^{h, \{r\}}( U, (t)\KKK[[t]]) \ge  \ivr_{k}^{h, \{r\}}( U, (t)\KKK[[t]])+  \ivr_{\ell}^{h, \{r\}}( U, (t)\KKK[[t]]) .
\end{equation} 
\end{thm}

\begin{proof}
This result is a special case of Theorem 1 of Bhargava \cite{Bhar:09}, applied to the local ring   $R= \KKK[[t]]$.
for the prime ideal $(t)\KKK[[t]]$, and says for the subset $U$ that 
$[k]!_{U}^{h, \{r\}} [\ell]!_{U}^{h, \{r\}} $ divides $[k +\ell]!_{U}^{h, \{r\}}$.
Then we have 
\begin{eqnarray*}
\ivr_{k+\ell}^{h, \{r\}}( U, (t)) -  \ivr_{k}^{h, \{r\}}( U, (t))-  \ivr_{\ell}^{h, \{r\}} ( U, (t))  &=& \nu_t( [k+\ell]!_U^{h, \{r\}}) -\nu_t([k]!_U^{h, \{r\}})- \nu_t([\ell]!_U^{h, \{r\}}) \\
&= & \nu_{t}( \frac{ [k+\ell]!_{U}^{h, \{r\}} }  {[k]!_{U}^{h, \{r\}} [\ell]!_U^{h, \{r\}} })
\ge 0,
\end{eqnarray*}
as asserted.
\end{proof} 

%%%%%%%%%%%%%%%%%%%%%%%%%%%%%%%%%%%%%%%%%%%%%%%%
% NOT DONE
%%%%%%%%%%%%%%%%%%%%%%%%%%%%%%%%%%%%%%%%%%%%%%%
%In Appendix \ref{Appendix:A}  we  provide an alternative proof of  Theorem \ref{thm:invariant}. 
%We prove a series of results which characterize the values 
%$\ddnu_k( \ddS,t \ddR[[t]], \mathbf{f})$ by  conditions not depending on $\mathbf{f}$.
%These conditions differ  from Bhargava's; they  are given as Max-min conditions,
%which we state formally in the following  Theorems \ref{thm:Max-min},  \ref{thm:Max-min-r},
%\ref{thm:Max-min-h}, and \ref{thm:Max-min-rh}.
%%%%%%%%%%%%%%%%%%%%%%%%%%%%%%%%%%%%%%%%%%%%%%%%

%%%%%%%%%%%%%%%%%%%
% Remark 4.3
%%%%%%
\begin{rem}\label{rem:43}
Theorem \ref{thm:T-invariant}  can be generalized  to
apply to  a larger class of formal power series rings $R= \DDS[[t]]$.
to show $\tp$-sequenceability
 for the maximal ideal $\mm= t \DDS$ (for various refined $\bb$-orderings). 
 When the coefficient ring $\DDS$ is not a field, 
 then $\DDS[[t]]$ will not be a Dedekind domain, which
 falls outside the Bhargava framework.
 
 We do not address the extent of possible generalization,
  noting only that  for ``plain" $t$-orderings 
 the class of allowed $\DDS$ will include all Dedekind domains $\DD$, 
 relying on generalization of Theorem \ref{thm:max-min} 
 valid for plain $t$-orderings for the maximal ideal $t \DD[[t]]$.  
\end{rem}

%%%%%%%%%%%%%%%%%%%%%%%%%%%%%%%%%%%%%%%%%%%%%%%%%%%%
%
% Subsec 4.8 :  Nondecreasing invariants (formerly A1
%
%%%%%%%%%%%%%%%%%%%%

\subsection{$t$-exponent sequences of refined $t$-orderings are non-decreasing}\label{subsec:48}

%%%%%%%%%%%%%%
% Lemma/Proposition  4.4
%%%%%%%%%%%%%%
\begin{prop}\label{prop:t-sequence-increasingA}
Let $h \in \NN \cup \{ +\infty\}$ and $r \in \NN$ be fixed. 
For nonempty $\ddS\subseteq\ddR[[t]]$ and a fixed
$r$-removed  $\tp$-ordering of order $h$, $\mathbf{f}=\left(f_i^{h, \{r\}}(t)\right)_{i=0}^\infty$,  of $\ddS$, 
the associated $\tp$-exponent 
sequence $\left(\ddnu_k^{h, \{r\}}(\ddS,\mathbf{f})\right)_{k=0}^\infty$
 is nondecreasing (=weakly increasing) for  $0 \le k < \infty$. That is, for$k \ge 0$,
 \begin{equation}\label{eq:weak-increase}
 \ddnu_{k+1}^{h, \{r\}}(\ddS,\mathbf{f}) \ge \ddnu_k^{h, \{r\}}(\ddS,\mathbf{f}),
 \end{equation}
\end{prop}

\begin{proof}
Let $\ddS\subseteq\ddR[[t]]$, and let $\mathbf{f}=\left(f_i^{h, \{r\}}(t)\right)_{i=0}^\infty$ be an
$r$-removed $\tp$-ordering of $\ddS$ of order $h$. 
Let $k$ be a nonnegative integer.
From the $t$-exponent sequence definition
%%%%%%%%%%
%( \eqref{eq:46} ) 
%%%%%%%%%%
associated to $\mathbf{f}$, we have
\begin{equation}\label{eq:A1}
\ddnu_k^{h, \{r\}}(\ddS,\mathbf{f}) :=\min_{\sA_{k, r}} \left[ \sum_{j \in \sA_{k,r}} \min[h, \dnu\left(f_k^{h, \{r\}}(t)-f_j^{h, \{r\}}(t)\right)] \right],
\end{equation}
Here $\sA_{k,r}$ runs over all $k-r$ element subsets of $\{0,1, \cdots , k-1\}$. We have similarly
\begin{equation}\label{eq:A2}
\ddnu_{k+1}^{h, \{r\}}(\ddS,\mathbf{f}) :=\min_{\sA_{k+1, r}} \left[ \sum_{j \in \sA_{k+1,r}} \min[h, \dnu\left(f_{k+1}^{h, \{r\}}(t)-f_j^{h, \{r\}}(t)\right)] \right],
\end{equation}
Now for  each  fixed $k+1-r$ element 
set $\sA_{k+1,r}$ of $\{0,1, \cdots k\}$, we set 
$\sA_{k,r}^{\ast} = \sA_{k+1,r} \smallsetminus \{\ell\},$ 
where $\ell$ is the largest element of $\sA_{k+1,r}.$ In particular $\sA_{k,r}^{\ast}$ is a $k-r$ element subset of $\{0, 1,..., k-1\}$ 
because 
the element $k$ is removed if it is in $\sA_{k+1, r}$. Then we have
\begin{eqnarray*}
 \sum_{j \in \sA_{k+1,r}} \min[h, \dnu\left(f_{k+1}^{h, \{r\}}(t)-f_j^{h, \{r\}}(t)\right)] 
&\ge & \left[ \sum_{j \in \sA_{k,r}^{\ast}} \min[h, \dnu\left(f_{k+1}^{h, \{r\}}(t)-f_j^{h, \{r\}}(t)\right)] \right]\\
&\ge& \min_{\sA_{k,r}}  \left[ \sum_{j \in \sA_{k,r}} \min[h, \dnu\left(f_{k+1}^{h, \{r\}}(t)-f_j^{h, \{r\}}(t)\right)] \right]\\
&\ge& \min_{\sA_{k,r}}  \left[ \sum_{j \in \sA_{k,r}} \min[h, \dnu\left(f_{k}^{h, \{r\}}(t)-f_j^{h, \{r\}}(t)\right)] \right]\\
&= & \ddnu_k^{h, \{r\}}(\ddS,\mathbf{f}),
\end{eqnarray*}
where the replacement of $f_{k+1}$ by $f_k$ in the third line holds by the definition of $f_{k^{h, \{r\}}(t)}$ being a minimizer 
over all elements of $U$ distinct from $\{ f_j^{h, \{r\}}(t): 0 \le j \le k-1\}$.
Next we  may minimize the left side over all sets $\sA_{k+1,r}$, noting the right side remains constant; 
the left side minimum is $\ddnu_{k+1}^{h, \{r\}}(\ddS,\mathbf{f})$, by \eqref{eq:A2}. 
Combining these results yields \eqref{eq:weak-increase},
%$$\ddnu_{k+1}^{h, \{r\}}(\ddS,\mathbf{f}) \ge \ddnu_k^{h, \{r\}}(\ddS,\mathbf{f}),$$
as asserted.
%Hence $\ddnu_k(\ddS,\mathbf{f})\le\ddnu_{k+1}(\ddS,\mathbf{f})-\dnu\left(f_{k+1}(t)-f_k(t)\right)\le\ddnu_{k+1}(\ddS,\mathbf{f})$.
\end{proof}

%%%%%%
% Remark 4.5
%%%%%%
\begin{remark}\label{rem:45}
Proposition \ref{prop:t-sequence-increasingA} implies that the (refined) generalized  integers $[n]_{\SSS, \T}^{h, \{r\}}$ are well-defined as integral ideals.
\end{remark}

%%%%%%%%%%
%%%%%%%%%%
%
% Section 5
%
%%%%%%%%%%
%%%%%%%%%%

\section{ Property $\PC$}\label{sec:propertyC}

The main result of the paper is a reduction from a commutative ring $\DR$ to
a formal power-series ring $\ddR[[t]]$ over a field $\ddR$ that preserves 
congruences $(\bmod \,\bb^k)$ to congruences $(\bmod \,(t)^k)$ for all $k \ge 1$.

%%%%%%%%%%
%
% Subsection 5.1
%
%%%%%%%%%%

\subsection{Mapping to ${\ddR[[t]]}$: Property C}\label{subsection:mapping}

Let $\RRR$ be a commutative ring and $\ddB$ a nonzero proper ideal of $\RRR$. We will show
that  a  sufficient condition 
 for  $\ddB$-sequencibility of $\RRR$ 
is the existence of a map $\varphi_\ddB:\RRR\rightarrow\ddR[[t]]$,
where $\ddR$ is 
%an integral domain, 
a field, that has the following property.\medskip

%%%%%%%%%%%%%%%%%%%%
%
%PROPERTY C
%
%%%%%%%%%%%%%%%%%%%%
\noindent\textbf{Property $\PC$.} \emph{For a commutative ring $\RRR$, we say a   map $\varphi_\ddB:\RRR\rightarrow\ddR[[t]]$ ,
where $\ddR$ is 
%an integral domain, 
a field, satisfies
 \emph{Property $\PC(\DR,\ddB,\ddR)$} (or simply \emph{Property $\PC$}) if for all $a_1, a_2 \in \RRR$,
\begin{equation}\label{map-condition}
\varphi_\ddB\left(a_1\right)\equiv\varphi_\ddB\left(a_2\right)\pmod{t^k}\quad\text{if and only if}\quad a_1\equiv a_2\pmod{\ddB^k}
\end{equation}
for all $k\ge1$. Equivalently,  for all $a_1, a_2 \in \RRR$,
\begin{equation}\label{eqn:equal-valuations}
\dnu\left(\varphi_\ddB\left(a_1\right)-\varphi_\ddB\left(a_2\right)\right)=\nuB\left(a_1-a_2\right).
\end{equation}
}

 \medskip
%%%%%%%%%%%%%%%%%%
%
% Remark 5.1
%
%%%%%%%%%%%%%%%%%%
\begin{rem}\label{rem:41}

(1) The identity \eqref{eqn:equal-valuations} 
that the map $\varphi_{\ddB}$
  is a congruence-preserving map between 
the rings $\RRR$ and $\ddR[[t]]$.
  The map $\varphi_{\ddB}(\cdot)$ may not  preserve  ring addition or ring  multiplication. 
That is, it need not be a homomorphism of the additive group structure, nor 
a semigroup morphism of
the multiplicative  monoid  $\RRR$.
 It need not preserve $0$ or $1$. 

(2)  Property  C does not require $R$ to be an integral domain. 
However if  $\bb$ is a nilpotent ideal, having  $\bb^{k}= (0)$, then  necessarily  the 
 map $\varphi$ must have  $\varphi_{\bb}(0)=0$. Consider $R= \ZZ/4\ZZ$, with $\ddB= (2)R$, and 
 Property C is satisfied with   $\ddR=\ZZ$, and $\varphi_{\ddB}(j + 4\ZZ) =j$ for $0 \le j \le 3.$

(2) The condition \eqref{map-condition} for all $k \ge 1$ implies   the map $\varphi_{\ddB}$ is injective, because
$\bigcap_{k=1}^{\infty} (t^k) \ddR[[t]] = (0)$.

\end{rem}

For $a\in\RRR$, denote by $\nuB(a)$ the supremum of all nonnegative integers $k$ such that $a\in\ddB^k$; i.e.,
$$
\nuB(a):=\sup\left\{k\in\mathbb{N}:a\in\ddB^k\right\} \bigcup \{+\infty\}.
$$
By convention  $\ddB^0=\RRR$. In addition $\nuB(0)= +\infty$ for all ideals.

%%%%%%%%%%%%%%%%
% Remark 5.2
%%%%%%%%%%%%%%%%
\begin{rem}\label{rem:52}
We have the  following set  of correspondences under the 
congruence-preserving map $\varphi_{\bb}$; the last two lines are justified in 
Proposition \ref{prop:main-propertyC}
below. The map $\varphi_{\bb}$ converts a semi-valuation $\order_{\bb}$ into a valuation.

%\newpage
%%%%%%%%%%%%%%%
%
% Table of NOTATIONS
%
%%%%%%%%%%%%%%%
\begin{table}[h]
\begin{center}
\begin{tabular}{|c|c|c|}
\hline
Ring  &$\RRR$&$\ddR[[t]]$\\
\hline
Subset&$\SSS$&$\ddS= \varphi_{\bb}(\SSS)$\\
\hline
Ideal &$\ddB$&$t\ddR[[t]]$\\
%\hline
%Additive (semi-)valuation& $\nuB$ & $\dnu$\\
%\hline
%Additive valuation&$~$&$\dnu$\\
\hline
$\bb$-ordering/ $t$-ordering &$\mathbf{a}=\left(a_i\right)_{i=0}^\infty$&$\mathbf{f}=\left(f_i(t)\right)_{i=0}^{\infty}=\left(\varphi_{\bb}(a_i)\right)_{i=0}^{\infty}$\\
\hline
Ordering  invariants&$\left(\ddnu_k^{h, \{r\}}\left(\SSS,\ddB \right)\right)_{k=0}^{\infty}$&$\left(\ddnu_k^{h, \{r\}}( \ddS,  t \ddR[[t]] )\right)_{k=0}^{\infty}$\\
\hline
\end{tabular}
\end{center}
\bigskip
\caption{Action of map $\varphi_{\bb}$}
\end{table}

\end{rem}

%%%%%%%%%%
%
% Subsection 5.2
%
%%%%%%%%%%

\subsection{Property $C(\RRR, \bb, \ddR)$
is sufficientf or $\bb$-sequenceability for all sets $\SSS$ for a fixed $\bb$}\label{section:B-orderings-definition}

%%%%%%%%%%
%
% Proposition  5.3
%
%%%%%%%%%%
\begin{prop}[$\ddB$-sequenceability of $\RRR$ via Property $\PC$]\label{prop:main-propertyC}
Given a commutative ring $\RRR$ and a nonzero proper ideal $\bb$ of $\RRR$.
Assume that $\varphi_\ddB:\RRR\rightarrow\ddR[[t]]$ satisfies Property $\PC(R, \bb, \ddR)$,
with $\ddR$ a field. 

(1)  For all  nonempty subsets $\SSS$ of $\RRR$, and   for all $\SSS$-test sequences $\mathbf{a}$, 
and  for all  $h \in \NN \cup \{\infty\}$, and all $r \in \NN$, the equality of $\SSS$-test-sequence invariants of $\mathbf{a}$
and $\varphi_{\ddB}(S)$-test sequence invariants of $\varphi_{\ddB}(\mathbf{a})$, 
 \begin{equation}\label{eqn:equal-inv}
\ivr_k^{h, \{r\}} \left(\SSS, \ddB, \mathbf{a}  \right)= \ivr_k^{h,\{r\}} \left( \varphi_{\ddB}(S), t\ddR[[t]], \varphi_{\ddB}(\mathbf{a})\right),
\end{equation} 
holds for all $k \ge 1$.

(2) 
$\RRR$ is $\bb$-sequenceable. 
%%%%%%%%%%%%%%%%%%%%%%%%%%%%%%%%%%%%%%%%%%%%%%%
%That is,  for any subset $\SSS \subset \RRR$ having at least two elements,
%any two associated $\bb$-orderings $\mathbf{a}_1$ and $\mathbf{a}_2$  of $\SSS$ have the same $\bb$-sequence.
 %That is,  $\ivr_i\left(S,b,\mathbf{a}_1\right)=\ivr_i\left(S,\bb,\mathbf{a}_2\right)$ for all $i=0,1,2,\dots$.
 %%%%%%%%%%%%%%%%%%%%%%%%%%%%%%%%%%%%%%%%%%%%%%%
 In addition, for  all nonempty  subsets $\SSS$ of $\RRR$,  all $h \in \NN \cup \{\infty\}$, and all $r \in \NN$,  
 \begin{equation}\label{eqn:equal-inv0}
\alpha_k^{h, \{r\}} \left(\SSS, \ddB \right)= \alpha_k^{h,\{r\}} ( \varphi_{\ddB}(S), ( t )\ddR[[t]]),
\end{equation} 
holds for all $k \ge 1$.
\end{prop}

%%%%%%%%%%%%%%%
%
% Proof of Proposition 5.2
%
%%%%%%%%%%%%%%%%%%
\begin{proof}
%[Proof of Proposition \ref{prop:main-propertyC}]

(1) Set $U=\varphi_{\ddB}(\SSS) \subseteq \ddR[[t]]$.
The map $\varphi_{\bb}(\cdot)$ maps $\SSS$-test sequences $\mathbf{a}$  for $\bb$ to $U$-test sequences for $t \ddR[[t]]$,\
where $U = \varphi_{\bb} (\SSS)$.
By definition for all $a_1, a_2 \in \SSS$, 
\begin{equation}\label{eqn:equal-valuations2}
\dnu\left(\varphi_\ddB\left(a_1\right)-\varphi_\ddB\left(a_2\right)\right)=\nuB\left(a_1-a_2\right).
\end{equation}
hence \eqref{eqn:equal-inv0} follows for all $\SSS$-test sequences.

(2) Now specialize to $\mathbf{a}^{h, \{r\}}$ being  a order $h$, $r$-removed
$\bb$-ordering of $\SSS$. 
The assertion that $\varphi_{\ddB}(\mathbf{a}^{h, \{r\}})$ is an  $r$-removed  $\tp$-ordering of 
$U=\varphi_{\ddB}(S) \subseteq \ddR[t]$ follows from Property $\PC(R, \bb, \ddR)$,
 since for all $a' \in S$ and all $j \ge 1$,   \eqref{eqn:equal-valuations} yields
\begin{equation*}
\dnu \left( \varphi_{\ddB}(a' ) - \varphi_{\ddB} (a_j^{h, \{r\}}) \right)=  \nuB \left( a' - a_j^{h,\{r\}}  \right).
\end{equation*}
Now $\ddR[[t]]$ is a Dedekind domain, so
by (Bhargava's) Theorem \ref{thm:T-invariant}  the ring $\ddR[[t]]$ is $\tp$-sequenceable, so we have:
\begin{eqnarray*}
\alpha_k^{h, \{r\}}(U ,  (t)   \ddR[[t]) &=& 
\alpha_k^{h, \{r\}}\left(\varphi_{\ddB}(S) ,  (t)   \ddR[[t]], \varphi_{\ddB}(\mathbf{a}^{h, \{r\}}) \right)  \\
&= &
\min_{\sA_{k,r}}\left[\sum_{j \in \sA_{k,r}} 
\min[h,\dnu\left(\varphi_{\ddB}(a_k^{h,\{r\}})-\varphi_{\ddB}(a_j^{h, \{r\}}\right)]  \right] \\
&=&\min_{\sA_{k,r}}\left[\sum_{j \in \sA_{k,r}} 
\min[h,\nuB\left(a_k^{h,\{r\}}-a_j^{h, \{r\}}\right)]  \right] \\
&=& \alpha_k^{h, \{r\}} (\SSS, \bb, \mathbf{a}^{h, \{r\}}),
\end{eqnarray*}
where the second equality used $\varphi_{\ddB}(\mathbf{a}^{h, \{r\}})$ being an  $r$-removed  $\tp$-ordering of $U$
and the 
 thind equality used Property $\PC(R, \bb, \ddR)$. 
%Thus  \eqref{eqn:independence-of-a} holds.

We have shown that
\begin{equation}\label{eqn:independence-of-a}
 \alpha_k^{h, \{r\}} (\SSS, \bb, \mathbf{a}^{h, \{r\}})= \alpha_k^{h, \{r\}}(U ,  (t)   \ddR[[t])
\end{equation} 
is independent of the choice of order $h$, $r$-removed $\bb$-sequence $\mathbf{a}^{h, \{r\}}$ of $\SSS$.
The equality holds for  all nonempty sets $\SSS$, therefore  the ring $\RRR$ is $\ddB$-sequenceable
 for  $r$-removed, order $h$ $\bb$-orderings.
The left side of \eqref{eqn:independence-of-a}  now  defines  $\alpha_k^{h, \{r\}} \left(\SSS, \ddB \right)$, hence
%\eqref{eqn:independence-of-a} yields
\begin{equation*}
\alpha_k^{h, \{r\}} \left(\SSS, \ddB \right)= \alpha_k^{h, \{r\}} ( U, ( t )\ddR[[t]])
\end{equation*} 
holds for all $k \ge 1$, concluding the proof of (2).
\end{proof} 

%%%%%%%%%%%%
% Remark 5.4
%%%%%%%%%%%%%%
\begin{rem}\label{rem:54}
The  definition of  congruence-preserving Property $\PC(\DR, \ddB, \DDS)$
 makes sense  for a target formal
power-series  ring $\DDS [[t]]$ whose coefficient ring  $\DDS$ is an arbitrary commutative ring.
For rings $\DDS$ for which the statement of Theorem \ref{thm:T-invariant} remains valid  (see Remark \ref{rem:43})
one would get a generalization of Proposition \ref{prop:main-propertyC} for $\ddB$-sequenceability
 allowing congruence-preserving maps having
Property $\PC(\DR, \ddB, \DDS)$.
%%%%%%%%%%%%%%%%%%%%%%%
% Following is moved to Remark 4.5
%%%%%%%%%%%%%%%%%%%%%%
% The ring $\DDS [[t]]$ is
%never a Dedekind domain when $\DDS$ is not a field, it  has  non-maximal prime ideals that are nonzero.
%In that case Bhargava's Theorem \ref{tm:B-invariant} cannot be applied.
%Nevertheless for  some rings $\DDS$ the maximal  ideal $\mm= t \DDS[[t]]$ will still be $\pp$-sequenceable.
%%%%%%%%%%%%%%%%%%%%%%%%% 
\end{rem}

%%%%%%%%%%
%
% Subsection 5.3
%
%%%%%%%%%%
\subsection{Properties of associated $\bb$-exponent sequences}\label{section:B-orderings-properties}

By Proposition \ref{prop:main-propertyC}  
$\bb$-exponent 
sequences inherit all the properties of 
%admissible 
$\tp$-exponent sequences. 
%We derive two consequences of Theorem \ref{prop:main-propertyC}.

%%%%%%%%%%
%
% Corollary 5.5
%
%%%%%%%%%%
\begin{corollary}[Local binomial coefficient inequality] \label{corollary:binomial-B}
Assume for a commutative  ring $\RRR$  with a nonzero proper ideal $\bb$ that
there is given a map $\varphi_\ddB:\RRR\rightarrow\ddR[[t]]$ satisfying 
Property $\PC(R, \bb, \ddR)$. Then for any subset $\SSS$ of $\RRR$ and any nonnegative integers $k$ and $\ell$, we have
$$
\ddnu_{k+\ell}^{h, \{r\}}\left(\SSS,\ddB\right)\ge\ddnu_k^{h, \{r\}}\left(\SSS,\ddB\right)+\ddnu_\ell^{h,\{r\}}\left(\SSS,\ddB\right).
$$
\end{corollary}

\begin{proof}
Applying Theorem \ref{thm:bhar-integral} with $\ddS=\varphi_\ddB\left(\SSS\right)$, we obtain the inequality
$$
\ddnu_{k+\ell}^{h, \{r\}}\left(\varphi_\ddB\left(\SSS\right), (t)\ddR[[t]] \right)\ge 
\ddnu_k^{h,\{r\}}\left(\varphi_\ddB\left(\SSS\right), (t) \ddR[[t]]\right)+\ddnu_\ell^{h,\{r\}}\left(\varphi_\ddB\left(\SSS\right), (t)\ddR[[t]] \right).
$$
By  \eqref{eqn:equal-inv} in Proposition \ref{prop:main-propertyC}  , the above yields 
$\ddnu_{k+\ell}^{h, \{r\}}\left(\SSS,\ddB\right)\ge\ddnu_k^{h,\{r\}}\left(\SSS,\ddB\right)+\ddnu_\ell\left(\SSS,\ddB\right)$.
\end{proof}

%%%%%%%%%%
%
% Corollary 5.6 (formerly 5.6)
%
%%%%%%%%%%
\begin{corollary}\label{cor:k-increasing}
Assume  for a commutative ring $\RRR$  with a nonzero  proper ideal $\bb$ 
 there is given a map $\varphi_\ddB:\RRR\rightarrow\ddR[[t]]$ satisfying Property $\PC(\RRR, \bb, \ddR)$. 
 Then for any subset $\SSS$ of $\RRR$ the four associated $\ddB$-exponent sequences 
 $\left(\ddnu_k \left(\SSS,\ddB\right)\right)_{k=0}^\infty$, resp. $\left(\ddnu_k^{\{r\}}\left(\SSS,\ddB\right)\right)_{k=0}^\infty$, resp. 
 $\left(\ddnu_k^{h}\left(\SSS,\ddB\right)\right)_{k=0}^\infty$, and resp. $\left(\ddnu_k^{h, \{r\}} \left(\SSS,\ddB\right)\right)_{k=0}^\infty$,
are all nondecreasing.
\end{corollary}

\begin{proof}
We indicate the proof for ``plain" $\bb$-exponent sequences; the same proof holds in the other three cases.
According to Lemma \ref{prop:t-sequence-increasingA} with $\ddS=\varphi_\ddB\left(\SSS\right)$, 
%we see that 
the associated $\tp$-sequence $\left(\ddnu_k \left(\varphi_\ddB\left(\SSS\right)\right)\right)_{k=0}^\infty$ is weakly increasing.
Now by  \eqref{eqn:equal-inv} in Proposition \ref{prop:main-propertyC}
this  sequence agrees term by term with  the  $\ddB$-sequence $\left(\ddnu_k\left(\SSS,\ddB\right)\right)_{k=0}^\infty$.
 giving the result.
\end{proof}

%%%%%%%%%%
%%%%%%%%%%
%
% Section 6
%
%%%%%%%%%%
%%%%%%%%%%
\section{Dedekind domains have  Property $\PC$}\label{sec:5} 

%We give 
%a direct proof of Property $\PC$ for the ring $\ZZ$, and then give
%a proof of Property $\PC$ for general Dedekind domains $\DD$. 

In this section we establish Property C for all nonzero proper ideals $\bb$ in a Dedekind domain $\DD$, 
one ideal at a time, using separate maps $\varphi_{\bb}(\cdot)$ for each  ideal $\bb$, see
Theorem \ref{thm:varphi_B-general}. 
 The special case  $\DD=\ZZ$ was previously  treated by the authors in\cite{LY:24a}.

%%%%%%%%%%
%
% Subsection .6.1
%
%%%%%%%%%%

\subsection{Digit expansions to base $\beta$
for a principal ideal $(\beta)$ in a  Dedekind domain}\label{subsec:proofs-DE}

We recall properties of  Dedekind domains $\DD$.   
A  {\em fractional ideal $\ff$}  of a Dedekind domain $\DD$ is any  $\DD$-module inside its quotient field $K$
having the property there is a nonzero element $\alpha \in \DD$
with $\frac{1}{\alpha} \ff \subset \DD$ is an ideal of $\DD$.
 Dedekind domains have  the property that 
all nonzero fractional ideals are invertible.
%%%%%%%%%%%%%%%%%%%%%%%%%%%%%%%%%%
%(They are characterized by this property among dimension one  integral domains,
% with some other axioms.
%%%%%%%%%%%%%%%%%%%%%%%%%%%%%%%%%%%
A Dedekind domain  has  
an associated ideal class group $\Cl(\DD)$,
which is the quotient of the  group of 
nonzero fractional ideals  by the group of principal nonzero fractional ideals.
 
 Dedekind domains
can be very large, there exist Dedekind domains of  any cardinality.
The ideal class group can be arbitrarily complicated;
 In 1966  Claborn (\cite{Clab:66})
 showed that every abelian group (of any cardinality)
can occur as the ideal class group of some Dedekind domain.  In 1975
Leedham-Green \cite{LeedG:72} showed  that every abelian group occurs
among class groups  of  Dedekind domains 
that are  quadratic extensions of a principal ideal domain.

%%%%%%%%%%%%%%%%%%%%%
%
% Lemma 6.1
%
%%%%%%%%%%%%%%%%%%%%%

\begin{lemma} \label{lem:Ded-congruence}
For any proper ideal $\bb$ in a Dedekind domain, there holds
%one has
\begin{equation*}
\bigcap_{k \ge 1} \bb^k = (0).
\end{equation*}
\end{lemma}

\begin{proof}
Let $a \in \DD$ be  an arbitrary nonzero element of $\DD$.  It suffices to show $a \not\in \bigcap_{k \ge 1} \bb^k.$
Let $a \DD= \prod_{j} \pp_j^{e_j}$ be the unique finite prime ideal factorization of $(a)$, and set $e = \max_{j} (e_j) < \infty.$
 Since $\bb$ is a nonzero proper ideal it is divisible by some prime ideal $\qq$, so $\bb^{e+1}$ is divisible by $\qq^{e+1}$.
 Now $\qq^{e+1} \nmid a \DD$, by unique prime factorization in Dedekind domains. 
 
 We recall that  ideals in Dedekind domains have the special property:  $\aaa_2 \mid \aaa_1$
 if and only if
 %%%%%%%%%%%%%%%%%%%%%%%%%%%%%%%%%%%%%%%%%%%%%%%%%%%%%%%%%%%%%
 %\footnote{(The ``if"  direction is clear. For the ``only if" set  $\cc = \{ x \in K: x \aaa_2 \subset \aaa_1\}$
 %where $K$ is the quotient field of $\DD$. Necessarily $\cc \subset \DD$ is an integral ideal, and by
% definition  $\aaa_2\cc \subset \aaa_1$. Equality holds since all  nonzero fractional ideals in a Dedekind domain are invertible.)}
%%%%%%%%%%%%%%%%%%%%%%%%%%%%%%%%%%%%%%%%%%%%%%%%%%%%%%%%%%%%%%%%%% 
   $\aaa_1 \subseteq \aaa_2$.  (For all commutative rings  $\aaa_2 \mid \aaa_1$
implies $\aaa_1 \subseteq \aaa_2$.) 
   
   Applying this equivalence (in contrapositive form: $\aaa_1 \not\subseteq \aaa_2$ if and only if $\aaa_2 \nmid \aaa_1$)
    we deduce $a \DD \not\subseteq \qq^{e+1}.$ By the same equivalence $\qq^{e+1} \mid \bb^{e+1}$ implies $\bb^{e+1} \subseteq \qq^{e+1}$.
   Now if we have   $a \in \bb^{e+1}$,  then $a \DD \subseteq \bb^{e+1}\subseteq \qq^{e+1}$,  a contradiction. We conclude
   $a \not\in \bb^{e+1}$, whence  $a \not\in \bigcap_{k \ge 1} \bb^k$.
\end{proof} 

The next lemma treats digit expansions for principal ideals of $\DD$.
%%%%%%%%%%%%%%%%%%%%%%%%%%%%%%%
%
% lemma  6.2
%
%%%%%%%%%%%%%%%%%%%%%%%%%%%%%%%
\begin{lemma}\label{lem:Ded-digit-exp}
Let $\DD$ be a Dedekind domain and $(\beta)$ a principal ideal in $\DD$. A  ``digit set" 
$\sD_{\beta} = \{ a_{[\alpha]} : \, [\alpha] \in \DD/ \beta \DD\} \subset \DD$  is  any  set of coset representatives $a_{[\alpha]} \in \DD$.
of the cosets $[\alpha]= \alpha_{[\alpha]}  + \beta \DD$ in the quotient ring $\DD/\beta \DD$. 
Consider a digit set $\sD_{\beta}$ having $d_{[0]}=0$ for the representative of the identity coset.
Then for each $a \in \DD$ there exists a  unique (formal) digit expansion $a= (d_0, d_1, d_2, \cdots)$, with each $d_j \in \sD_{\beta}$,
specified by the infinite system
of congruences $(\bmod \, \beta^k)$, 
\begin{equation}\label{eqn:digit-exp}
a \equiv \sum_{i=0}^{k-1} d_i \beta^i \,\, (\bmod \, \beta^k),
\end{equation}
holding for all $k \ge 1$.
\end{lemma}

%Conversely, the  digit expansion uniquely determines $a \in \DD$.
{\em Remark.}
We denote 
%existence of 
 a digit expansion for $a \in \DD$ by: $a \sim \sum_{k=0}^{\infty} d_j \beta^j$ having each $d_i \in \sD_{\beta}$.
There is no converse: in general,  an expansion with given digits
 $(d_1, d_1, d_2,\cdots) \in \sD_{\beta}^{\NN}$ might not correspond to any element of $\DD$.

\begin{proof}
For existence of an expansion \eqref{eqn:digit-exp}, we proceed by induction on $k \ge 1$, assuming as induction hypothesis 
the existence for $k$ of a partial expansion of form
$$
a = \sum_{i=0}^{k-1} d_i \beta^i  + a_{k} \beta^k
$$
with $a_k \in \DD$. 

Given $a \in \DD$ pick 
$d_0 \in \sD$ to be the unique coset representative in $\sD$ of $a+ \DD$. 
Then we set 
$a_1^{'} := a- d_0 \in \beta \DD$, and we may write  
 $a_1^{'}= \beta a_1,$ for a unique $a_1 \in \DD$. Then
 $ a_0 = d_0 + a_1 \beta$, which gives  the base case $k=1$ of the induction. 
 
 For the induction step we are given the remainder term $a_k\beta^k$.
 %and we may write   $a_k= a_{k+1} \beta$, 
 Choose digit $d_{k}\in \sD$ to be the coset representative in $a_k +\DD$, 
 set $a_{k+1}^{'} := a_k - d_k \in \beta \DD$ and define
 $a_{k+1}$ uniquely by requiring  $a_{k+1}^{'} = a_{k+1} \beta$. Then we have
 $$
 a = \sum_{i=0}^{k} d_i \beta_i + a_{k+1} \beta^{k+1},
 $$
 completing the induction step.
  
The uniqueness of the expansion \eqref{eqn:digit-exp}
follows from Lemma \ref{lem:Ded-congruence}.
If $a_1$ and $a_2$ have the same expansion, then $a_1 \equiv a_2 \, (\bmod \, \beta^k)$ for all $k \ge 1$,
whence $a_1- a_2 \in \cap_{k \ge1} \beta^k \DD =\{0\}.$
\end{proof} 

%%%%%%%%%%%%%%%%%%%
%
% Subsection .6.2 PID case
%
%%%%%%%%%%%%%%%%%%%

\subsection{Property C holds for principal ideal domains}\label{subsec:proofs-PID}

%The next lemma   constructs 
We construct a suitable map $\varphi_{\bb}: \DD \to \DD[[t]] \subset K [[t]]$ giving Property $\PC(\DD, \bb, \ddR)$
in the case that $\bb = \beta \DD$ is a principal ideal of the Dedekind domain, using a 
power-series digit expansion 
to base $\beta$. For  $\DD=\ZZ$, such an expansion was given constructively
in the proof of Proposition 5.1
of  \cite{LY:24a}. 

%the proof of Proposition [??]
%\ref{prop:varphi_B-exists-for-Z}  
%in \cite{LY:23a}
%for $\DD=\ZZ$   given via  a digit expansion
 %to principal ideals in a Dedekind domain.
%given in Proposition \ref{prop:varphi_B-exists-for-Z}.

%%%%%%%%%%%%%%%%%%%%%%%%%%%%%%%
%
% lemma  6.3
%
%%%%%%%%%%%%%%%%%%%%%%%%%%%%%%
\begin{lemma}\label{lem:Ded-digit-exp2}
For a Dedekind domain $\DD$ with quotient field $K$ and principal ideal $(\beta)$, Property $\PC(\DD, (\beta), K)$ holds for  
%the map 
$\varphi_{(\beta)}: \DD  \to \DD[[t]] \subset \KKK[[t]]$ given by 
$\varphi_{(\beta)}(a) = \sum_{k=0}^{\infty} d_k t^k$
where $a \sim \sum_{k=0}^{\infty} d_k \beta^k$ is the digit expansion with respect to  a fixed digit set $\sD$
given by Lemma \ref{lem:Ded-digit-exp}.
\end{lemma}

\begin{proof}
The constructed map $\varphi_{(\beta)}$ will have  image in $\DD[[t]]$,
but for Property $\PC(\DD, \beta, \KKK)$  this map is viewed as having the larger codomain  $K[[t]]$.
% for the purpose of Property  $\PC(\DD, (\beta), K)$.

 Let $d_k(a)$ denote the $k$-th digit of the expansion of $a \in \DD$ using the digit set $\sD$. 
Then, for any $a_1, a_2 \in \DD$, and each $k \ge 1$,  
\begin{eqnarray*}
a_1 \equiv a_2 \, (\bmod \, (\beta^k) ) & \Leftrightarrow& d_i(a_1) = d_i(a_2) \quad \mbox{for} \quad 0 \le a < k,\\
& \Leftrightarrow& \varphi_{(\beta)}(a_1) - \varphi_{(\beta)}(a_2) \equiv 0 \, (\bmod \, t^k \DD[[t]]) \\
& \Leftrightarrow& \varphi_{(\beta)}(a_1) - \varphi_{(\beta)}(a_2) \equiv 0 \, (\bmod \, t^k K[[t]]).
\end{eqnarray*}
The last equivalence holds since $t^k \DD[[t]] = t^k K[[t]] \cap \DD[[t]$.
%the image of $\varphi_{\beta}$ lies in $\DD[[t]].$
The last equivalence verifies  property $\PC(\DD, (\beta), K).$
%( This assertion directly generalizes  the proof in the case $\DD=\ZZ$ based on a digit expansion
%given in Proposition \ref{prop:varphi_B-exists-for-Z}. )
\end{proof}

%%%%%%%%%%%%%%%%%%%
%
% Subsection 6.3 Property C Dedekidn General  case
%
%%%%%%%%%%%%%%%%%%%

\subsection{Property C holds for Dedekind domains}\label{subsec:proofs-general}

To treat  non-principal ideals in  the Dedekind domain $\DD$, we use an alternative  construction,
which works for nonzero proper ideals $\bb$.
It is based on  a result, due to  Claborn \cite[Proposition 2.6]{Clab:65a} in 1965, stating that  any Dedekind domain $\DD$
can be  embedded in a larger Dedekind domain $\DD^{'}$ which is a principal ideal domain.

This construction uses localizations of the polynomial ring $\DD[X]$.
The notation 
\begin{equation}
 \DD[X]_{S} := (S)^{-1} \DD[X]
 \end{equation}
  denotes  localization at a multiplicatively closed set $S \subseteq \DD[X] \smallsetminus \{(0)\}$.
We let  $S_1$  be the set of all monic polynomials in $\DD[X]$,
and we let $S_2$ to be the larger set of all primitive polynomials in $\DD[X]$, 
which are those  with coefficients in $\DD$ that have content ideal $(1)$.
Recall that a   polynomial  $f(X) = \sum_{i=0}^n c_i X^{i} \in D[X]$  (with $c_n \ne 0$) has {\em content ideal}  
\begin{equation*}
\ct (f) : = (c_0, c_1, ..., c_n) \DD \subset \DD.
\end{equation*}
For a  Dedekind domain  $\DD$ the content ideal  of $f \in \DD[X]$ 
is the greatest common divisor ideal of the ideals $(c_i)$ over 
$0 \le i \le n$.  (This ideal may be non-principal.)
For Dedekind domains a  basic  property of content ideals (\cite[Lemma 2.4]{Clab:65a}),
%%%%%%%%%%%%%%%%%%%%%%%%%%%%%%%%%%%%%%%%
% \cite[Chapert 4.2]{Lang:2002}) is less general, needs factorial ring
%%%%%%%%%%%%%%%%%%%%%%%%%%%%%%%%%%%%%%%%
  is multiplicativity,  
\begin{equation}\label{eqn:content-mult}
\ct(f_1 f_2) =\ct(f_1) \ct(f_2) \quad \mbox{for all} \quad  f_1, f_2 \in \DD[X].
\end{equation} 
The multiplicativity property  implies  that $S_2$ is multiplicatively closed. 

We have
$$
D[X] \subset D[X]_{S_1} \subset D[X]_{S_2} \subset K(X),
$$
where $K$ is the fraction field of $D$ and $K(X)$ is the one-variable function field over $K$.
%The ring $\DD^{'} = \DD(X) := \DD[X]_{S_2}$ is a localization of the polynomial ring $\DD[X]$ at
% the multiplicatively closed subset $S_2$ of
%all primitive polynomials, which are those  with coefficients in $\DD$ that have content ideal $(1)$.
The  notation $\DD(X)$  for   $\DD[X]_{S_2}$ is that used in  the 1962 book of 
%the 1962 book of 
Nagata \cite[Sec. 6,  (6.13)ff]{Nag:62}. Viewed inside
the rational function field $K(X)$, the ring $\DD(X)=\DD[X]_{S_2}$ is given by
\begin{equation}\label{eqn:def-DX}
\DD(X) := \{ r(X) \in K(X) :  r(X) = \frac{f(X)}{g(X)}, \, f, g \in \DD[X], \, \, \ct(g)=1\}.
\end{equation}
 
%The set $S_1$ of all monic polynomials is  
%multiplicatively closed by inspection, and $S_1 \subset  S_2$. 

%%%%%%%%%%%%%%%%%%%%%%%%%%%%%%%
%
% Proposition 6.4
%
%%%%%%%%%%%%%%%%%%%%%%%%%%%%%%%
\begin{prop}[Claborn] \label{prop:Claborn}
For any  Dedekind domain $\DD$ the ring $\DD^{'}=\DD[X]_{S_2} $
is a Dedekind domain and is a principal ideal domain.
%Consequently the   injective
The  ring homomorphism
$\psi: \DD \to \DD^{'}$ is injective, and 
 takes  each   nonzero proper ideal $\bb$ of $\DD$ to an  extended ideal $\bb^{'} = \psi(\bb) \DD'$
that is a principal ideal. Here $\psi(\bb) = \{ \psi(a) : a \in \bb\}$.
\end{prop}

\begin{proof}
Claborn \cite[Lemma 2-1]{Clab:65a} showed that if $\DD$ is a Dedekind domain, then $\DD[X]_{S_1}$
is a Dedekind domain. Therefore  the further localization  
%uses the notation $\DD(X) := \DD[X}_{S_2}$. 
 $\DD(X) := \DD[X]_{S_2}$ is automatically a Dedekind domain.   In \cite[Proposition 2-6]{Clab:65a} 
 Claborn showed that  $\DD(X)$  is a principal ideal domain.
 The rest of the proposition immediately follows. 
 
 Claborn \cite[Proposition 2--3]{Clab:65a} additionally
 showed  that 
 % localization $\DD{X}_{S_1}$ at the  the monic polynomials $S_1$ is that that 
the inclusion map $\iota: \DD \to \DD[X]_{S_1}$ induces  an isomorphism of ideal class groups,
$\iota: \Cl(\DD) \to \Cl(\DD[X]_{S_1})$.  
The further localization  inverting $S_2$ then  annihilates the  ideal class group. 
 \end{proof}

We identify  $\DD \subset \DD(X)$, using the injective map $\psi: D \to D(X)$.
We need further properties of the  injective map $\psi: \DD \to \DD[X]_{S_2}$.
In showing these properties  we can obtain another proof of Proposition \ref{prop:Claborn},
see Remark \ref{rem:66}.

%%%%%%%%%%%%%%%%%%%%%%%%%%%%%%%
%
% Proposition  6.5
%
%%%%%%%%%%%%%%%%%%%%%%%%%%%%%%%
\begin{prop}\label{prop:Claborn2}
Let $\DD$ be a Dedekind domain with fraction field $K$, and let
$\DD(X) = \DD[X]_{S_2} \subset K(X)$.

(1) For each ideal $\aaa$ of $D$, the extended ideal
$\aaa^{e} = \aaa \DD(X)$ under the inclusion $\iota: \DD \to \DD(X)$ is a principal ideal.

(2) Each ideal $\aaa$ of $D$ satisfies
\begin{equation}\label{eqn:ec-Ded2}
\aaa = \aaa \DD(X) \cap \DD.
\end{equation}
Equivalently, the induced homomorphism $\iota_{\aaa}: \DD/ \aaa \to \DD(X)/ \aaa \DD(X)$
is injective.
\end{prop}

\begin{proof}  For each ideal $\aaa$ of $D$ we set
\begin{equation}\label{eq:extend-ideal-frac}
I_{\aaa} := \{ r(x) \in K(X) :  r(x) = \frac{f(x)}{g(x)}, \, f, g \in \DD[X], \ct{g}=1, \, \, \aaa \mid \ct{f}\},
\end{equation}
It is an ideal of $\DD(X)$, and $I_{(0)} = 0$, $I_{\DD(X)} = \DD(X)$.
We can extend the notion of content to all members of $\DD(X)$, setting $\ct(r) := \ct(g)$ in
the representation $\frac{f(x)}{g(x)}$  above. 
This representation is not unique, but its content is well-defined by
\eqref{eqn:content-mult}. The ideal $I_{\aaa}$ 
consists of  all elements of $K(X)$ having content $\aaa$ in this sense.

 (1) We show all image ideals $I_{\aaa}$  in $\DD(X)$ are principal.  The zero ideal is principal, so we may assume $\aaa$ is nonzero.
 Every nonzero ideal $\aaa$ of a Dedekind domain is generated by two nonzero elements $\alpha_0, \alpha_1$
 where the first element can be chosen arbitrarily  (\cite[Chap. 13, Cor. 4.5]{AusB:74});  
Thus we may write $\aaa= (\alpha_0, \alpha_1) \DD.$ We claim that $\aaa^{e} :=\aaa \DD(X)$ has
\begin{equation}\label{eq:ideal-match} 
\aaa^{e} = ( \alpha_0 + \alpha_1 x) \DD(X)  = I_{\aaa}.
\end{equation} 
Here $\ct(\alpha_0+ \alpha_1x) = (\alpha_0, \alpha)\DD= \aaa$.
We first observe the inclusions 
$$
( \alpha_1 + \alpha_2 x) \DD(X)\subseteq \aaa \DD(X) \subseteq  I_{\aaa},
$$
where the leftmost inclusion holds since
 $\alpha_1, \alpha_2 x \in \aaa \DD(X)$, and the rightmost inclusion holds since $\aaa \subset I_{\aaa}$
 and $I_{\aaa}$ is a $\DD(X)$-ideal.
 
 To prove \eqref{eq:ideal-match} it suffices to show that $ I_{\aaa} \subseteq ( \alpha_1 + \alpha_2 x) \DD(X)$.
To this end  consider an arbitrary nonzero element  $r(x) = \frac{f(x)}{g(x)} \in I_{\aaa}$, with  $\ct(g)=1$ and  $\aaa \mid \ct(f)$. 
We have
$$
r(x) = (\alpha_0 + \alpha_1 x) \tilde{r}(x) \quad \mbox{with} \quad \tilde{r}(x) = \frac{f(x)}{g(x)(\alpha_0 + \alpha_1x)} \in K(X).
$$
We will show $\tilde{r}(x) \in D(X)$, which certifies $r(x) \in (\alpha_0 + \alpha_1 x) \DD(X)$. Pick an integral  ideal $\bb$ of $D$ in
the (fractional) ideal class of $D$  inverse to $\aaa$, so that $\aaa \bb = (\gamma)$ for some $\gamma \in \DD$. Let $\bb= (\beta_0, \beta_1)\DD$
be nonzero generators of $\bb$, whence
 $\ct(\beta_0+ \beta_1 x ) = \bb$.  
 Now we have $\ct( (\alpha_0+\alpha_1 x)(\beta_0+ \beta_1x)) =  \aaa \bb = (\gamma)$.
 It follows that  $h(x) :=\frac{1}{\gamma}(\alpha_0+\alpha_1 x)(\beta_0+ \beta_1x)  \in \DD[X]$ has $\ct(h(x))= (1).$
Additionally set $\tilde{f}(x) = f(x) (\beta_0 + \beta_1 x)$ Then $\ct(\tilde{f}) = \ct(f) \ct(\beta_0+ \beta_1x)$ is divisible by
$\aaa\bb = (\gamma)$. We conclude that $k(x)= \frac{1}{\gamma} f(x) (\beta_0 + \beta_1 x) \in \DD(X).$ Now we have
$$
\tilde{r}(x)= \frac{f(x)(\beta_0+ \beta_1x)}{g(x)(\alpha_0 + \alpha_1x)(\beta_0+ \beta_1x)}= \frac{k(x)}{g(x) h(x)},
$$
where both $g(x) h(x)$ and $k(x)$ are in  $\DD(X)$, and the  content $\ct(gh) = \ct(g)\ct(h)=(1)$. Thus
$\tilde{r}(x) \in D(X)$.
We conclude from \eqref{eq:ideal-match} that $ \aaa^c= \aaa\DD(X) = (\alpha_0+ \alpha_1 x) \DD(X)$ is a principal ideal.

%Given an ideal $\cc$, we let $\ct(I)$ be the
%gcd of the contents of all its members. One can then show that the ideal $I$ will contain a single element
%having content $I$, 

(2) We have $I_{\alpha} \cap \DD = \aaa$ from its definition \eqref{eq:extend-ideal-frac}. 
Since $\aaa^e= I_{\alpha}$ by \eqref{eq:ideal-match} the result \eqref{eqn:ec-Ded2} follows. 
The equivalence follows from observing  $\ker(\iota_{\aaa}) = (\aaa \DD(X))/ \aaa.$
\end{proof} 

%%%%%%%%%%
%
% Remark 6.6
%
%%%%%%%%%
\begin{rem}\label{rem:66}
One can further show  in (1) above that 
the ideals $I_{\aaa}$, running over the ideals $\aaa$ of $\DD$, 
comprise  the complete set of (integral) ideals of $\DD(X)$.
 Given an ideal $\cc$ of $\DD(X)$, the associated  $\aaa$ in $\cc= I_{\aaa}$ will be its content $\ct(\cc)$,
defined as the gcd of the contents of all $r(x) \in \cc$. To show the full set of ideals occurs
one must then verify that  every ideal  of content $\ct(\cc)$
contains a  polynomial $f$  with $\ct(f)= \cc$.
\end{rem}

%%%%%%%%%%
%
%  Theorem 6.7
%
%%%%%%%%%%
\begin{thm}[Property C maps for ideals of Dedekind domains]\label{thm:varphi_B-general}
Given a Dedekind domain $\DD$,
for each nonzero proper ideal $\bb$ in $\DD$ there is an injective map
$$\varphi_{\bb}: \DD \to \DD^{'} [[t]] \subset K^{'}[[t]],$$
 which satisfies
Property $\PC( \DD,\bb, K^{'})$.
Here   $\DD^{'}=\DD(X)$  is a Dedekind domain ( given in \eqref{eqn:def-DX}),
and  $K^{'}$ is the quotient field of $\DD^{'}$.
\end{thm}

\begin{proof} 
We are given a proper  ideal $\bb$ of $\DD$. By Proposition \ref{prop:Claborn}
 under the embedding $\DD \to \DD' =\DD(X)$ the extension
ideal  $\bb' =\bb \DD^{'}:=  \bb \DD(X)$ is a principal ideal $\bb' = \beta\DD(X)$  for some $\beta \in \DD(X)$.
Since $\DD^{'}$ is a Dedekind domain, we may
 use a digit expansion of elements of $\DD(X)$ modulo $\bb'$ given by Lemma \ref{lem:Ded-digit-exp}.
 We write $\bb' = \beta \DD^{'}$, making an arbitrary choice of generator $\beta \in \DD^{'}$ as a principal ideal. 
  We make an arbitrary choice of  a digit set $\sD \subset \DD^{'}$  for all  cosets $x+ \beta \DD^{'}$.
 (Note that the element $\beta$ and the digits $d_i$ are generally {\em not} in the image of $\DD$ inside $\DD^{'}$.) 
 Then  for  each 
 $a' \in \DD^{'}$  we obtain  a (unique) expansion
 $$
 a' \sim \sum_{i=0}^{\infty} d_i \beta^i,
$$
with all $d_i \in \sD$, satisfying
\begin{equation}\label{eqn:b-prime-congruence}
a' \equiv \sum_{i=0}^{k-1} d_i \beta^{i} \, (\bmod \, (\bb')^{k}), 
\end{equation} 
for all $k \ge 1$.
We  define the map $\varphi_{\bb'}: \DD^{'} \to \DD^{'} [[t]]$ into the power series ring $\DD^{'}[[t]]$ by
$$
\varphi_{\bb} (a') := \sum_{i=0}^{\infty} d_i t^{i}. 
$$
The map $\varphi_{\bb'}$ is injective by the uniqueness property in Proposition \ref{prop:Claborn}.

We  verify that  the map $\varphi_{\bb}$ satisfies Property $\PC( \DD,\bb,  K^{'}).$
We already know by Lemma \ref{lem:Ded-digit-exp2} that 
 $\varphi_{\bb'}$ has image in $\DD'[[t]]$ and satisfies Property $\PC(\DD', \bb', K^{'})$, where we define
$\varphi_{\bb'}: \DD^{'} \to \DD^{'} [[t]]$ by sending $\alpha \sim d_i \beta^{i} \in \DD^{'}$ to 
$\varphi_{\bb}(\alpha) = \sum_{i=0}^{\infty} d_i t^{i} \in \DD^{'} [[t]]$.
% Here $\bb' = (\beta) \DD^{'}$ 
%and this  fact  follows from Lemma \ref{lem:Ded-digit-exp}
%taking $D = \DD^{'}$ in the lemma. For it implies
%\begin{equation} 
%a_ \equiv a_2 (\bmod (\beta^k) \DD^{'})  \Longleftrightarrow \varphi_{\bb'} (a) \equiv \varphi{\bb'}(a_2) \, (\bmod \, (t^k)\DD^{'} [[t]]). 
%\end{equation}
Thus $\DD'$ is  $(\beta) \DD^{'}$-sequenceable.

The map $\varphi_{\bb}: \DD \to \DD^{'} [[t]]$ is obtained by restriction of the map $\varphi_{\bb'}$ to $\DD$, viewing
$\DD$ as contained in $\DD'[[t]]$.  Here $\bb'= \beta' \DD^{'}$ is principal.
The assertion  that $\varphi_{\bb}$ satisfies 
Property $\PC( \DD,\bb, K^{'})$ 
% \DD^{'})$ 
follows by restriction of the domain
of $\varphi_{\bb'}$  to $\DD$. 
%%%%%%%%%%%%%%%%%%%%%%%%%%%%
%because  Proposition \ref{lem:Claborn2} implies 
%$\bb^k= (\bb')^k \cap \DD$ by \eqref{eqn:ec-Ded2},
%and because the map $\varphi_{\bb'}(\cdot)$ is injective. 
%%%%%%%%%%%%%%%%%%%%%%%%%%%%%
Namely  consider a fixed  set $S \subset \DD \subset \DD^{'}$.
 That is, for $a_1, a_2 \in S$, and setting $\ddS = \varphi_{\bb'}(S)$, for $a_1, a_2 \in S$, 
\begin{eqnarray*}
  a_1  \equiv a_2\,  (\bmod \, \bb^k)  & \Longleftrightarrow& a_1  \equiv a_2\,  (\bmod \, (\bb^{'})^k \cap \DD) \\
      & \Longleftrightarrow& a_1  \equiv a_2\,  (\bmod \, (\beta^k \DD') ) \\
%      & \Longleftrightarrow& d_j (a_1 ) = d_j( a_2) \,  \quad \mbox{for} \quad 0 \le j < k\\
& \Longleftrightarrow& \quad \varphi_{\bb'} (a_1) \equiv \varphi_{\bb'}(a_2) \, (\bmod \, (t^k)\DD^{'} [[t]]) \\
& \Longleftrightarrow& \quad \varphi_{\bb'} (a_1) \equiv \varphi_{\bb'}(a_2) \, (\bmod \, (t^k)K^{'} [[t]]). 
\end{eqnarray*} 
The  equivalence on the first line holds by Proposition  \ref{prop:Claborn2} (2) via \eqref{eqn:ec-Ded2},
that on the second line holds because $a_1, a_2 \in D$ by hypothesis and $( \bb^{'})^k= \beta^k \DD'$,
and that on the third line holds by Proposition $\PC(\DD', \bb', K^{'})$, with the image being in $\DD'[[t]]$.
The last equivalence holds since $t^k \DD'[[t]] = t^k K^{'}[[t]] \cap \DD^{'}[[t]$.
We conclude that  Property $\PC( \DD,\bb,  K^{'})$ holds. By Theorem \ref{prop:main-propertyC} we conclude
$\DD$ is  $\bb$-sequenceable. 
\end{proof}

%%%%%%%%%%%%%%%%%%%%%%
%
%  Remark  XXXXX5.9a
%  (omit for now: fluff
%
%%%%%%%%%%%%%%%%%%%%
%\begin{rem}\label{rem:57}
%For composite ideals $\bb$ it seems believable
%that  the map $\varphi_{\bb'}$ cannot be a ring homomorphism,
%because we are mapping elements of a  non-prime ideal $\beta\DD'$
%to those of a  maximal ideal $t \DD^{'}[[ t]]$.)
%\end{rem} 
%%%%%%%%%%%%%%%%%%%
%
%  Subsection 6.4 Proof of Main Theorem 3.6
%6
%
%%%%%%%%%%%%%%%%%%%
\subsection{Proof of Theorem ~\ref{thm:NEW-MAIN}}
%THEOREM 3.6
%\ref{thm:well-definedness}}\label{subsec:proofs-53}

We prove the well-definedness Theorem \ref{thm:NEW-MAIN}
 and discuss the
construction. Separate proofs are given for the zero ideal $\bb =(0)$ and the unit ideal  $\bb= (1)$.
%= \DD$.

%%%%%%%%%%%
% Proof  of Thm 3.6
%%%%%%%%%%%%%
\begin{proof}[Proof~of~Theorem~\ref{thm:NEW-MAIN}]
%\ref{thm:well-definedness}]
 Given a general Dedekind domain $\DD$ and an ideal $\bb$, we first apply Theorem \ref{thm:varphi_B-general}
 to conclude there is a an injective map $\varphi_{\bb}: \DD \to \DD^{'} [[t]] \subset \KKK'[[t]]$,
 with $\DD^{'}=\DD(X)$ a Dedekind domain, 
 which satisfies
Property $\PC( \DD,\bb,  \KKK').$  
Now  Theorem \ref{prop:main-propertyC}
yields for all  subsets $\SSS$ of $\RRR$ and for all $\SSS$-test sequences $\mathbf{a}$, that
 \begin{equation}\label{eqn:equal-inv1}
\ivr_k^{h, \{r\}} \left(\SSS, \ddB, \mathbf{a}  \right)= \ivr_k^{h,\{r\}} \left( \varphi_{\ddB}^{'}(S), ( t )\KKK^{'}[[t]], \varphi_{\ddB}^{'}(\mathbf{a})\right).
\end{equation} 
In particular, we now obtain that $\RRR$ is $\bb$-sequenceable, as a consequence of Bhargava's Theorem \ref{thm:Bhar-main09}
applied taking $\DD= \KKK^{'}[[t]]$ and $\pp= (t) \KKK^{'}[[t]]$, 
for all  $r$-refined $\bb$-orderings $\aaa$ of order $h$, for any $0 \le r < \infty$ and $0\le h \le + \infty$.

%%%%%%%%%%%%%%%%%%%%%%%%%%%
% Exceptional cases $\bb=(0)$ and $\bb=(1)=\DD$.
%%%%%%%%%%%%%%%%%%%%%%%%%%%

It remains to treat the cases $\bb=(0)$ and $\bb=(1)=\DD$.  
We verify the assertions 
for these  $\bb$-orderings  and $\bb$-sequences directly in $\DD$, without using a map with
Property C.  

For the case $\bb=(1) =\DD$, we have $\ordr_{(1)}( a- a^{'}) = +\infty$
for all elements $a, a' \in \DD$, which implies that
all $\SSS$-test sequences $\mathbf{a}$ are $\bb=(1)$-orderings.  For  $n \ge r+1$ we compute that
\begin{equation}\label{eqn:1-sequence}
\ivr_n^{h, \{r\}}(\SSS, (1), \mathbf{a} ) = \sum_{i \in A_{n,r}} \min\left(h, \ordr_{(1)}(a_n- a_i)\right) = h(n-r),
\end{equation}
while $\ivr_n^{h, \{r\}}(\SSS, (1), \mathbf{a} )=0$ for $0 \le n \le r$. 
In all cases, for all $n \ge 0$, 
$$
\nu_{n}( \SSS, (1)) = (1)^{\ivr_n^{h, \{r\}}( \SSS, (1), \mathbf{a} ) } = (1).
$$

The case $\bb=(0)$ is more complicated. We have  $\ordr_{(0)}( a- a^{'}) = 0$ if $a \ne a'$
and $\ordr_{(0)}( 0)= +\infty$. For $n=0$ we  have $\ivr_{0}(\SSS, (0), \mathbf{a} ) = 0$, by convention. 
For $n \ge1$, initial  $\bb=(0)$-orderings are those that avoid having $(r+2)$ equal elements as long as possible.
This possibility can be postponed  for $1 \le n < |S| (r+1)$, independent of the value of $h$. Therefore all
such sequences 
\begin{equation}
\ivr_n^{h, \{r\}}(\SSS, (0), \mathbf{a} ) = 0 
\end{equation} 
for $0\le n < (r+1)\vert \SSS \vert$, and
for $n= (r+1)\vert \SSS \vert$ we  have (for any $h$) 
\begin{equation}
\ivr_n^{h, \{r\}}(\SSS, (0), \mathbf{a} ) = 1.
\end{equation} 
All larger $n$ must have an $r$-removed $a_n$ agreeing with an earlier element,  
so it follows that 
\begin{equation*} 
\nu_{n}( \SSS, (0), \mathbf{a}) = (0)^{\ivr_n^{h, \{r\}}( \SSS, (0),\mathbf{a} ) } = \begin{cases} 
 (1) & \mbox{for} \quad 0 \le n < (r+1) \vert \SSS \vert,\\
 (0)  & \mbox{for} \quad \quad n \ge  (r+1) \vert \SSS \vert,
 \end{cases}
 \end{equation*}
 using the convention $(0)^0=1.$
For fixed $(h, r)$, it can be checked that  the associated $\bb=(0)$ sequences
are,  for $n \ge  (r+1) \vert \SSS \vert$, 
\begin{equation}\label{eqn:general-r-h-zero}
\ivr_n^{h, \{r\}}(\SSS, (0)) = h \left\lceil  \frac{ n+1- (r+1) \vert \SSS \vert}{ \vert \SSS \vert} \right\rceil,
\end{equation} 
where $\lceil \cdot \rceil$ is the ceiling function. 
We omit details characterizing all $\bb=(0)$-orderings 
and verifying  independence of the $\bb$-ordering, for $n \ge (r+1)\vert \SSS \vert$.
%%%%%%%%%%%%%%
%\Details could be added}
%%%%%%%%%%%%%%%%

If $\vert \SSS \vert$ is an infinite set, then the ideals associated to $\bb=(0)$ and $\bb=(1)$ are the
full ring $(1)$ for all $n \ge 0$. 
\end{proof}

%%%%%%%%%%
%
%  Remark 6.8
%
%%%%%%%%%%
\begin{rem}\label{rem:610}
The ring $\DD(X)$ in Theorem \ref{thm:varphi_B-general} has the property that for  each of its nonzero proper ideals $\bb \DD(X)$ the 
ring $\DD(X)/ \bb\DD(X)$ has infinite cardinality. In particular, the  residue
classes $T^k + \bb\DD(X)$ for $k \ge 0$ are all distinct $(\bmod \, \bb\DD(X))$
because   the difference of any two residue classes 
is $T^k -T^j \, (\bmod \, \bb\DD(X))$ and $\ct (T^k- T^j) =(1)$ when $j \ne k$, so $T^k -T^j$ is a unit in $\DD(X)$,
  so not in $\bb\DD(X)$.

%If  one chooses a ``random" infinte subset 
It follows that  if $\DD$ is a Dedekind domain, then any subset $\SSS$ of  the Dedekind domain $\DD(X)$ that hits infinitely
many different residue classes $(\bmod \, \bb)$ necessarily has  
 associated
$\bb$-exponent sequence $\alpha_k(S, \bb)=0$  for all $k \ge 0$. 
Thus  only  very special subsets $\SSS \subset \DD(X)$ 
can have $\bb\DD(X)$-exponent sequences having at least one exponent
taking a  nonzero value. However it  is {\em exactly} such special subsets
that are constructed and used  in the proof above. 
\end{rem} 

%%%%%%%%%%%%%%%%%%%%%%%%%%%%%%%%%%%
%
% Remark 6.9
%
%%%%%%%%%%%%%%%%%%%%%%%%%%%%%%%%%%%%%
\begin{rem}\label{rem:69}

We can give an alternate proof of Theorem \ref{thm:NEW-MAIN}  without appeal to Bhargava's Theorem \ref{thm:Bhar-main09},
for ``plain" $t$-orderings.
%%%%%%%%%%%%%%%%%%%%%%%%%%
% Proof  of Thm 3.6 (Plain B-orderings case)
%%%%%%%%%%%%%%%%%%%%%%%%
\begin{proof}[Alternate~ proof~of~Theorem~\ref{thm:NEW-MAIN} (plain $t$-orderings)] 
For   the special case of plain $t$-orderings ($r=0, h= +\infty$), we may 
 use Theorem \ref{thm:varphi_B-general}
 to obtain the  congruence-preserving injective map $\varphi_{\bb}: \DD \to \DD^{'} [[t]]\subset \KKK'[[t]]$.
To complete the argument, at  the last step we appeal  to Theorem \ref{thm:max-min}
giving  well-definedness of $t$-sequence invariants  for $\KKK^{'}[[t]]$ with $\pp= (t) \KKK^{'}[[t]]$
independent of $t$-ordering, via a max-min criterion.
\end{proof}

\end{rem}

%%%%%%%%%%
%%%%%%%%%%
%
% Section 7
%
%%%%%%%%%%
%%%%%%%%%%
\section{Generalized factorial 
ideals  and  generalized  binomial coefficient ideals}\label{sec:6}

We  define generalized-integer ideals, generalized factorial ideals, and
generalized binomial coefficient ideals   for $\sI$-sequenceable rings $\RRR$,
via factorizations allowing $\bb$-sequences for  all ideals $\bb$ of $\RRR$.  These ideals
are initially defined as fractional ideals. 
We  prove  results stated in Section \ref{sec:2}, deducing that generalized-integer ideals and generalized binomial
coefficient ideals are integral ideals of $\DD$.

%%%%%%%%%%%%%%%%%%%
%
% Subsection 7.1 Generalized Factorials 
%
%%%%%%%%%%%%%%%%%%%
\subsection{Product formulas  for generalized factorials}\label{subsec:62} 

We give formulas expressing  generalized-integer ideals and generalized binomial coefficient ideals
 as  (infinite) products.  
 As a first step below we   prove a finiteness lemma showing that generalized factorials are well defined,
i.e. only finitely many terms are not the unit ideal.
 
 A priori these products define  fractional ideals; we  establish 
 in later lemmas that  they  are integral ideals.
 The individual factors in these product formulas include  composite ideals. 
 Therefore obtaining  complete  prime factorizations inside the Dedekind domain
requires factorizing  composite ideals and combining  the prime ideal
contributions of different terms.

%%%%%%%%%%
%
% lemma  7.1 (new)
%
%%%%%%%%%%
\begin{lem}\label{lem:60} 
Let $S$ be a nonempty subset of the Dedekind domain $\DD$. Then 
for positive integers $n$ with $1 \le n<\vert S\vert$, there are only finitely many ideals $\bb \in \sI(\DD)$ such
that $\ivr_i^{h, \{r\}}(S, \bb) >0$ for some $i$ with $1 \le i \le n-r$.

In particular,  the generalized factorial
\begin{equation}\label{eqn:factorial-def}
[k]!_{S,\T}^{h, \{r\}} =\prod_{\bb \in\T}\bb^{\ivr_k^{h, \{r\}}(S,\bb)}
\end{equation}
is well-defined for $0 \le k \le \vert S \vert -r-1$.
\end{lem}

\begin{proof} Since $n < |S|$ we may pick  $n+1$ distinct elements
$a_0, a_1, \cdots, a_{n-1}, a_n \in S$, which defines the initial part of an $S$-test sequence 
${\bf a} = (a_0, a_1, \cdots, a_n)$. 
%%%%%%%%%%%%%%%%%%%%%%%%%%%%%%%%%%%%%%%%%%%%%%%%
%$S'= \{a_0, a_1, \cdots, a_{n}\}$. 
%Letting $b_{i, j} = a_i - a_j$ for each $i \ne j$ with $0 \le i,j \le n$ (there are $n(n+1)$ choices)
%%%%%%%%%%%%%%%%%%%%%%%%%%%%%%%%%%%%%%%%%%%%%%%%
There are only finitely many ideals $\bb'$ in $\sI(\DD)$ such that $\bb' \mid a_i- a_j$ for some $0 \le i < j \le n$,
since each  principal ideal $(a_i-a_j)$ has a  finite unique factorization into prime ideals (up to order of the prime ideal factors);
if this factorization has  $k$  prime ideals, then  has at most $2^k$ distinct ideal divisors $\bb'$.
For all other ideals $\bb \in \sI(\DD)$, the $S$-test sequence ${\bf a}$ has 
all its  $\bb$-exponents $\ivr_i^{h, \{r\}}(S, \bb, {\bf a})= 0$
for $1 \le i \le n-r$. 
This fact  certifies that the $S$-test sequence ${\bf a}$  is the initial part of an order $h$, $r$-removed,  $\bb$-ordering for $S$.
Therefore  $\ivr_i^{h, \{r\}}(S, \bb)= 0$ for $1 \le i \le n-r$, proving the finiteness. 

The finiteness certifies the right side of \eqref{eqn:factorial-def} has only finitely many $\bb \in \sI(\DD)$ having strictly positive exponents,
so it is a finite factorization. 
\end{proof}

%%%%%%%%%%
%
% lemma  7.2
%
%%%%%%%%%%
\begin{lem}\label{lem:61} 
Let $S$ be a nonempty subset of the Dedekind domain $\DD$, and let  $\T\subseteq \sI(\DD)$. Then 
for positive integers $n<\vert S\vert-r$,
the generalized-integer ideals $[n]_{S,\T}$ are given by
\begin{equation}\label{eqn:integer-factorization-formula}
[n]_{S,\T}^{h, \{r\}}=\prod_{\bb\in\sT} \bb^{\gamma_n(S, \bb; h; \{r\})},
\end{equation}
in which
\begin{equation}
\gamma_n( S, \bb; h, \{r\})   :={\ivr_n^{h, \{r\}}(S,\bb)-\ivr_{n-1}^{h, \{r\}}(S,\bb)}.
\end{equation}
They are integral ideals.
\end{lem}

\begin{proof}
This formula (viewed as a fractional ideal) follows from Definition \ref{definition:integers-S-T} and Definition  \ref{definition:factorials-S-T}.
However we now know each term $b^{\ivr_n^{h, \{r\}}(S,\bb)-\ivr_{n-1}^{h, \{r\}}(S,\bb)}$
  is an integral ideal since 
$$
\ivr_n^{h, \{r\}}(S,\bb)-\ivr_{n-1}^{h, \{r\}} (S,\bb)\ge 0
$$
holds for all $n \ge 1$ by Theorem \ref{thm:varphi_B-general} combined with 
Corollary \ref{cor:k-increasing}. 
%%%%%%%%%%%%%%
%Note that $\bb=(0)$ and $\bb=(1)$ contribute ideals that are either $(1)$ or $(0)$ and 
%will not change integrality.
 %%%%%%%%%%%%%%%

The  right side of 
the formula \eqref{eqn:integer-factorization-formula} is well-defined by appeal to the
finiteness Lemma \ref{lem:60}. 
\end{proof}

%%%%%%%%%%
%
% lemma 7.3
%
%%%%%%%%%%
\begin{lem}\label{lem:62}
Let $S$ be a nonempty subset of the Dedekind domain  $\DD$. Let $\T\subseteq\sI(\DD)$. Then for integers 
$0\le\ell\le k<\vert S\vert-r$,
the generalized binomial coefficients  ${k\brack\ell}_{S,\T}$  are given by
\begin{equation}\label{eqn:binomial-factorization-formula}
{k\brack\ell}_{S,\T}^{h, \{r\}}=\prod_{\bb\in\T} \bb^{ \ivr_k^{h,\{r\}}(S,\bb)-\ivr_\ell^{h, \{r\}}(S,\bb)-\ivr_{k-\ell}^{h, \{r\}}(S,\bb)}.
\end{equation}
They are integral ideals. 
\end{lem}

\begin{proof}
This formula (viewed as a fractional ideal)
follows from Definition \ref{definition:binomial-S-T} and Definition \ref{definition:factorials-S-T}.
However we now know each term $\bb^{\ivr_k^{h, \{r\}}(S,\bb)-\ivr_\ell^{h, \{r\}}(S,\bb)-\ivr_{k-\ell}^{h, \{r\}}(S,\bb)}$
 is an integral ideal since 
$$
\ivr_k^{h, \{r\}}(S,\bb)-\ivr_\ell^{h, \{r\}}(S,\bb)-\ivr_{k-\ell}^{h, \{r\}}(S,\bb) \ge 0
$$
holds for all $k \ge \ell \ge 0$, using   Theorem \ref{thm:varphi_B-general} combined with 
Corollary \ref{corollary:binomial-B}.
 
The right side of the  formula \eqref{eqn:binomial-factorization-formula} to be well-defined, 
by appeal to the finiteness Lemma \ref{lem:60}. \end{proof}

%%%%%%%%%%%%%%%%%%%
%
%
% Subsection 7.2 Proofs of Section 2
%
%
%%%%%%%%%%%%%%%%%%%
\subsection{Proofs of inclusion  and integrality properties in  Section \ref{sec:2} }\label{subsec:61} 

We prove Proposition \ref{prop:factorial-ordering}, Theorem \ref{thm:integer-integer} and
Theorem \ref{thm:binomial-integer}.
%%%%%%%%%%%%%%%%%%%
%
%Proof of Prop. 3.10
%
%%%%%%%%%%%%%%%%%%%
\begin{proof}[Proof~of~Proposition~\emph{\ref{prop:factorial-ordering}}]
  We have $[k]_{S, \sT_2}^{h, \{r\}}$ is divisible by  $[k]!_{S, \sT_1}^{h, \{r\}}$
%This is true
 because
$$
[k]!_{S,\T_2}^{h, \{r\}}=[k]!_{S,\T_2\backslash\T_1}^{h, \{r\}}[k]!_{S,\T_1}^{h, \{r\}}
$$
and generalized factorials $[k]!_{S, \T}$ are integral ideals in $\DD$.
\end{proof}
%%%%%%%%%%%%%%%%%%%
%
% Proof of Proposition 3.12
%
%%%%%%%%%%%%%%%%%%%

\begin{proof}[Proof~of~Theorem~\emph{\ref{thm:integer-integer}}]
It follows from Theorem \ref{thm:varphi_B-general} and Corollary \ref{cor:k-increasing} that
$$
\ivr_n^{h, \{r\}}(S,\bb)\ge\ivr_{n-1}^{h, \{r\}}(S,\bb)\quad\mbox{for all}\quad\bb\in\sI(\DD).
$$
The result then follows from \eqref{eqn:integer-factorization-formula}.
\end{proof}

%%%%%%%%%%%%%%%%%%%
%
% Proof of Theorem 3.15
%
%%%%%%%%%%%%%%%%%%%

\begin{proof}[Proof~of~Theorem~\emph{\ref{thm:binomial-integer}}]
It follows from Theorem \ref{thm:varphi_B-general} and Corollary \ref{corollary:binomial-B} that
$$
\ivr_{k+\ell}^{h, \{r\}}(S,\bb)\ge\ivr_k^{h, \{r\}}(S,\bb)+\ivr_{\ell}^{h, \{r\}}(S,\bb)\quad\mbox{for all}\quad\bb\in\sI(\DD).
$$
The result then follows from \eqref{eqn:binomial-factorization-formula}.
\end{proof}

%%%%%%%%%%%%%%%%%%%%
%
% Section 8
%
%%%%%%%%%% %%%%%%%%%%
\section{Concluding remarks}\label{sec:8}

\begin{enumerate}
\item[(1)]
 This paper  establishes $\bb$-sequenceablity  for all nonzero proper ideals $\bb$ of 
Dedekind domains $\DD$. One may consider more generally {\em Dedekind-type rings},
which are rings of the form  $\RRR= \DD/ \aaa$, for some ideal $\aaa$, which may have
zero divisors. 
It is known that if $\aaa$ is a proper ideal, then any such $R$, which is  a
semi-local ring,  inherits  unique factorization of ideals; moreover it is a principal ideal domain.
Which ideals $\bb$ of such a ring are $\bb$-sequenceable? 
It is possible to use  Condition C to show $\bb$-sequenceability of all proper ideals in  some  rings, e.g. $\ZZ/4\ZZ$
 in  Remark  \ref{rem:41} (2).

\item[(2)]
The proofs of this paper depend on   
an appeal to Bhargava's Theorem \ref{thm:T-invariant}  
giving well-definedness of his refined $t$-invariants for a local ring, used as a black box. 
Independent proofs of  invariance of the individual refined $t$-orderings
in  Theorem \ref{thm:T-invariant}  would yield independent proofs of our main result.
 In Theorem 4.3 of \cite{LY:24a} we
gave  an independent proof covering the plain $t$-ordering case, which is 
 $r=0, h= +\infty$. 
 To state the result,  we say that polynomial $p(x;t)$ of degree $k$  in $x$, with coefficients in $\ddR[[t]]$, 
is \emph{$t$-primitive} if $\dnu\left(\left[x^i\right]p(x;t)\right)=0$ for some  $i$, $0 \le i \le k$.

 %%%%%%%%%%%%%%%%%
%
%  Theorem  8.1 max-min plain
%
% (Thm 4.3 in [LY23a]
%%%%%%%%%%%%%%%%
\begin{thm}[Max-min characterization of plain $t$-exponent sequences]\label{thm:max-min}
Let $\ddS$ be a nonempty subset of a formal power series ring $\KKK[[t]]$, for $\KKK$ a field.
Let $\mathbf{f}=\left(f_i(t)\right)_{i=0}^\infty$ be a $\tp$-ordering of $\ddS$.
Then
\begin{equation}\label{equation:nu=maxmin0}
\ddnu_k(\ddS,  t \ddR[[t]], \mathbf{f})=\max_{p(x;t)\in \Prim_k} \left[ \,\min_{f(t) \in \ddS} \left\{\dnu(p(f(t);t))\right\}\right],
\end{equation}
where the maximum runs over the set $\Prim_k$ of all $t$-primitive polynomials $p(x;t)$ of degree 
at most $k$ in $x$.
\end{thm}

This result supplies  a max-min condition characterizing the invariants,
which manifestly demonstrates  independence of the $t$-ordering.
The proof of this result given in \cite{LY:24a} assumes  $\ddR= \QQ$, 
but  it remains  valid for fields $K$ of any characteristic.
It shows that the maximum in
\eqref{equation:nu=maxmin0}
is attained for the primitive polynomial
$$
q_k(x;t):=\prod_{j=0}^{k-1}\left(x-f_j(t)\right) \in \Prim_k,
$$
which has  degree exactly $k$, so that, 
\begin{equation}\label{eqn:equality-qk}
\ddnu_k(\ddS,t \ddR[[t]] \mathbf{f})=\min_{f(t) \in \ddS} \left\{\dnu\left(q_k(f(t);t)\right)\right\}.
\end{equation}
We have also obtained such a max-min criterion  valid for $r$-removed $p$-orderings (with $h= +\infty$, unpublished).

\item[(3)]
Proposition \ref{prop:factorial-ordering} establishes  a divisibility property (resp. inclusion relation) of generalized factorials
for arbitrary $(h, \{r\})$-orderings on sets $\SSS$ of a Dedekind domain. 
on varying the set of ideals $\T$ while holding the set $\SSS$ fixed.

Concerning parallel divisibility properties (resp. inclusion relations), 
for variation of the set $\SSS$ while holding $\sT$ fixed, in Proposition 5.4 (3) of \cite{LY:24a} we showed that, for $D= \ZZ$,
if $S_1 \subset S_2$ and $\T \subset \sI^{\ast}(\ZZ)$ is fixed, then for ``plain" $\T$-orderings ($r=0$, $h= +\infty$), 
and $0 \le k <|S_1|$, 
\begin{equation}\label{eq:inclusion-S}
[k]!_{S_2,\T}  \quad\text{divides}\quad[k]!_{S_1,\T},
\end{equation}
equivalently, regarded as ideals,  
\begin{equation}\label{eq:inclusion-S2}
[k]!_{S_1,\T}  \quad\text{is contained in}\quad[k]!_{S_2,\T}.
\end{equation}
These results are proved  invoking the max-min criterion Theorem \ref{thm:max-min}.

Since  Proposition 5.4 for $\DD=\ZZ$ of \cite{LY:24a} generalizes 
 for  Dedekind domains $\DD$, 
 Theorem \ref{thm:max-min} applies to give 
  divisibility  relations \eqref{eq:inclusion-S} and inclusion relations \eqref{eq:inclusion-S2} also hold
for  ``plain" $\bb$-orderings for arbitrary Dedekind domains $\DD$. 

The divisibility results  can be extended further  to apply for  refined $\bb$-orderings, which are $r$-removed for any $r \ge 0$,
 but of  order $h=+\infty$,
using  a generalized version of the max-min criterion for $r$-removed orderings 
mentioned in (2). 
It remains to determine whether a parallel  divisibility property (resp. inclusion relation)  
  holds in the general case varying $\SSS$ and holding $\T$ fixed, for arbitrary $(h, \{r\})$-orderings
for Dedekind domains.

%%%%%%%%%%%%%%%%%%%%%%%%%%%%%%%%%%%%%%%%%
% This comment moved to be mentioned  in Introduction, just before Section 1.1 starts.
%%%%%%%%%%%%%%%%%%%%%%%%%%%%%%%%%%%%%%%%%%
%\item[(4)]
%This paper  defined generalized factorial ideals built with products over 
%fixed subsets $\T$ of nonzero integral ideals $\sI(\DD)$. In the case where $\T$ is a 
%subset of the prime ideals $\Spec(\DD)$, various of the resulting factorials 
%are known to have interesting combinatorial
%interpretations, see the examples of Bhargava \cite{Bhar:00} for $\DD=\ZZ$. 
%It remains a topic for further investigation. 
%to find such structural interpretations for the generalized factorial ideals with $\T= \sI(\DD)$.
%%%%%%%%%%%%%%%%%%%%%%%%%%%%%%%%%%%%%%%%%%%%%% 

\item[(4)]
 In \cite{DLY:25+} the  authors, together with Lara Du,
 study for $\DD=\ZZ$ properties of the  
 %partial factorizations 
 products of the generalized binomial coefficients
$$
\prod_{\ell=0}^n {n\brack\ell}_{\mathbb{Z},\sI(\ZZ)}
$$
on the n-th row of the corresponding
Pascal triangle. The properties concern the 
partial factorizations of these products,  taken over all the primes $p \le \alpha n$ for
fixed $\alpha$ with $0 < \alpha \le 1$, letting $n \to \infty$.
Their results are compared with results on partial factorizations of the
usual binomial coefficients, carried out in \cite{DL:22}. The latter results
relate to the Riemann hypothesis.

The present paper allows not to  
 define for Dedekind domains $\DD$ generalized factorials and generalized binomial
coefficients, in which,  where  for the $n$-th factorial, in  the
upper index of a binomial coefficient the set  $\sT= \sT_n$ is allowed to (monotonically) 
vary  with $n \ge 1$. 
When the ring $\DD$ has a norm function on ideals, one can 
study the sizes of such partial factorizations,  
as $n \to \infty$.
%\medskip
\end{enumerate}

%The error terms in
%the asymptotics of the logarithms of  these partial factorizations have error terms
%with  power savings over the main term, assuming
%the Riemann hypothesis. It seems likely that an unconditional  remainder  term in that
%problem with a power savings would imply a zero-free region for the Riemann zeta function in a
%strip $\Re(s) > 1- \delta$ for some positive $\delta$.
%In contrast, the results  \cite{DLY:22+} show that 
%one can obtain asymptotic formulas having a power savings unconditionally
%in the remainder term.\medskip
% these generalized binomial coeffiients..

\paragraph{\bf Acknowledgments} 
The first author was partially supported by NSF grant DMS-1701576. The second author
was partially supported by NSF grant DMS-1701577.

%%%%%%

%%%%%%%%%%%%%%%%%%%%
%
% References
%
%%%%%%%%%% %%%%%%%%%%

\end{document}